\documentclass[11pt]{article}

\usepackage{amsfonts,amsmath,amsthm,amssymb}
\usepackage{mathabx,bm}
\usepackage{verbatim}
\usepackage{hyperref,url}
\usepackage{graphicx}
\usepackage{color}
\usepackage{soul}
\usepackage{tikz}
\usetikzlibrary{shapes}
\usepackage{ifthen}
\usepackage[numbers]{natbib}
\usepackage[margin=1in]{geometry}

\newtheorem{theorem}{Theorem}[section]
\newtheorem{lemma}[theorem]{Lemma}
\newtheorem{corollary}[theorem]{Corollary}

\newtheorem{conjecture}[theorem]{Conjecture}

\newcommand{\LL}{\mathcal{L}}
\newcommand{\R}{\mathcal{R}}
\newcommand{\D}{\mathcal{D}}
\newcommand{\N}{\mathcal{N}}
\newcommand{\Nhat}{\widehat{\mathcal{N}}}
\newcommand{\Ncheck}{\widecheck{\mathcal{N}}}

\begin{document}

\title{Solving Abalone on Small Boards}
\author{J. Gutstadt \thanks{ Egen Solutions, {jgutstadt20@gmail.com} }\and K. Hogenson \thanks{Skidmore College, {khogenso@skidmore.edu}} \and J. Koerner \thanks{Corresponding author. Ascend Analytics, {koernerjw@gmail.com} }}

\maketitle

\begin{abstract}
   \noindent Abalone is a 2-player board game with perfect information. The game is played on a $5\times5\times5$ hexagonal grid and ends when a player pushes 6 of their opponents' pieces off the board. Abalone is similar to games like chess and Go in that all three games have high branching factors, making it difficult for a computer to determine the outcome of a game. However, solving smaller, simplified versions of Abalone can offer insight into how to play the full-size game optimally.  In this paper, we strongly solve a variation of Abalone played on a $2\times2\times2$ hexagonal board.  We also weakly solve an Abalone variation on a $2\times2\times3$ hexagonal board.
\end{abstract}


\section{Introduction}
\label{sec:intro}

Abalone is a board game designed by Michel Lalet and Laurent Levi in 1987. It was made available to the public in September 1988 and won ``Game of the Decade'' in 1998 from the Festival International des Jeux \cite{GamesCrafters}.  It also won a Mensa Select award in 1990 and was Spiel des Jahres Recommended in 1989 \cite{BoardGameGeek}.

Since Abalone is a 2-player game with perfect information and alternating turns, it bears similarity to other well-known games like chess, Go, and checkers. These games are well-studied, but the large state spaces and branching factors for chess and Go make even weak solutions difficult to attain. Instead, theorists have focused on developing Artificial Intelligence (AI) agents to play these games and on studying simplified game variants. Abalone has seen similar treatment in the AI department \cite{aichholzer2002algorithmic, campos, chorus, claussen, hutson, Lemons, ozcan2004simple, papa, verloop}, but until now, no simplified variants of Abalone have been considered.  In this paper, we will begin to explore the mathematics behind Abalone by considering small and non-regular Abalone boards.

\subsection{Gameplay}
\label{subsec:gameplay}

Reminiscent of Chinese checkers, the Abalone game board is a regular hexagonal grid consisting of 61 spaces with 5 spaces on each side. There are two players: Left, who plays with black marbles, and Right, who plays with gray marbles. For clarity, we will use Black and Gray to refer to these players for the remainder of the paper.  Each player starts with 14 marbles that they take turns moving. The players can move 1 to 3 contiguous marbles each turn, in either a line or broadside motion. The marbles must move into empty spaces unless they are moving in a line where the length of the line outnumbers the length of the opponent's line occupying the desired spot. In this situation, the player with a longer line can push their opponent's marbles in the direction of the motion (3 marbles can push 1 or 2 marbles, 2 marbles can push 1).  The objective of the game is to push 6 of the opponent's marbles off the edge of the board.  See Figure \ref{fig:movingandpushing} for examples of broadside movement, in-line movement, and a \emph{sumito}, or push, move.

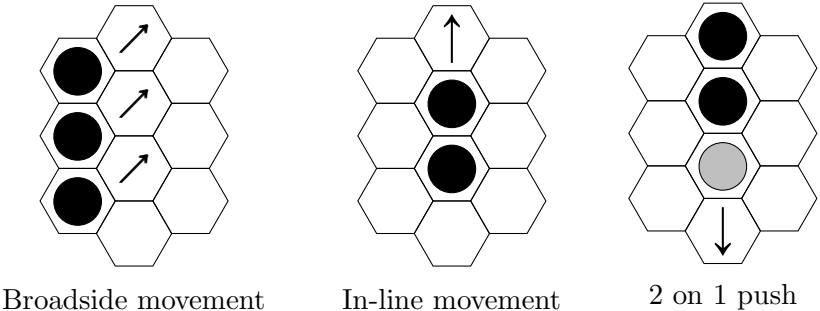
\begin{figure}[ht!]
\caption{Legal moves and pushes.}
\label{fig:movingandpushing}
\centering
\vspace{.5cm}
    \begin{tikzpicture}[hexa/.style= {shape=regular polygon,
                                   regular polygon sides=6,
                                   minimum size=1cm, draw,
                                   inner sep=0,anchor=south,}]
    \foreach \j in {0,...,2}{%
     \ifodd\j 
         \foreach \i in {0,...,3}{\node[hexa] at ({\j/2+\j/4},{(\i+1/2)*sin(60)}) {};}        
    \else
         \foreach \i in {1,...,3}{\node[hexa] at ({\j/2+\j/4},{\i*sin(60)}) {};}
    \fi}
    \draw[fill=black] (0,{(2.5)*sin(60)}) circle (9pt);
    \draw[fill=black] (0,{(1.5)*sin(60)}) circle (9pt);
    \draw[fill=black] (0,{(3.5)*sin(60)}) circle (9pt);
    \node at (.75,{2*sin(60)}) {\large $\boldsymbol\nearrow$};
    \node at (.75,{3*sin(60)}) {\large $\boldsymbol\nearrow$};
    \node at (.75,{4*sin(60)}) {\large $\boldsymbol\nearrow$};
    \node at (.75,0) {Broadside movement};
    
    \end{tikzpicture}
    \hspace{.5cm}
    \begin{tikzpicture} [hexa/.style= {shape=regular polygon,
                                   regular polygon sides=6,
                                   minimum size=1cm, draw,
                                   inner sep=0,anchor=south,}]
    \foreach \j in {0,...,2}{%
     \ifodd\j 
         \foreach \i in {0,...,3}{\node[hexa] at ({\j/2+\j/4},{(\i+1/2)*sin(60)}) {};}        
    \else
         \foreach \i in {1,...,3}{\node[hexa] at ({\j/2+\j/4},{\i*sin(60)}) {};}
    \fi}
    \draw[fill=black] (0.75,{(2)*sin(60)}) circle (9pt);
    \draw[fill=black] (0.75,{(3)*sin(60)}) circle (9pt);
    \node at (.75,{4*sin(60)}) {\huge ${\uparrow}$};
    \node at (.75,0) {In-line movement};
    
    \end{tikzpicture}
    \hspace{.5cm}
    \begin{tikzpicture} [hexa/.style= {shape=regular polygon,
                                   regular polygon sides=6,
                                   minimum size=1cm, draw,
                                   inner sep=0,anchor=south,}]
    \foreach \j in {0,...,2}{%
     \ifodd\j 
         \foreach \i in {0,...,3}{\node[hexa] at ({\j/2+\j/4},{(\i+1/2)*sin(60)}) {};}        
    \else
         \foreach \i in {1,...,3}{\node[hexa] at ({\j/2+\j/4},{\i*sin(60)}) {};}
    \fi}
    \draw[fill=black] (0.75,{(4)*sin(60)}) circle (9pt);
    \draw[fill=black] (0.75,{(3)*sin(60)}) circle (9pt);
    \draw[fill=lightgray] (0.75,{(2)*sin(60)}) circle (9pt);
    \node at (.75,{1*sin(60)}) {\huge ${\downarrow}$};
    \node at (.75,0) {2 on 1 push};
    
    \end{tikzpicture}
\end{figure}

There are multiple accepted starting positions in Abalone. The standard setup, seen on the left in Figure~\ref{fig:5x5x5board}, lines up each players' marbles on opposite sides of the board. This is the setup described in the board game rules, but many tournaments opt to use alternate configurations such as the Belgian daisy, which can be seen on the right in Figure~\ref{fig:5x5x5board}.  Conventional knowledge surrounding the game suggests that the standard configuration allows for interminable defensive strategies, while the Belgian daisy encourages more aggressive play and leads to fewer draws \cite{Wikipedia}. 

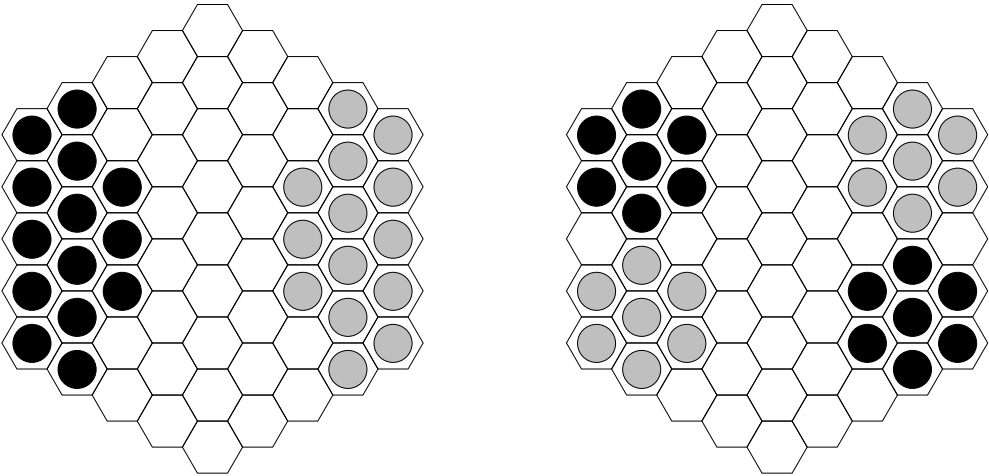
\begin{figure}[ht!]
\caption{The standard initial configuration (left) and Belgian daisy (right).}
\label{fig:5x5x5board}
\begin{minipage}[b]{0.5\linewidth}
\centering
\vspace{.5cm}
\scalebox{0.8}{
    \begin{tikzpicture} [hexa/.style= {shape=regular polygon,
                                   regular polygon sides=6,
                                   minimum size=1cm, draw,
                                   inner sep=0,anchor=south,}]
    \foreach \j in {0,...,8}{%
        \ifthenelse{\j=0 \OR \j=8}{\foreach \i in {2,...,6}{\node[hexa] at ({\j/2+\j/4},{\i*sin(60)}) {};}}
            {\ifthenelse{\j=1 \OR \j=7}{\foreach \i in {1,...,6}{\node[hexa] at ({\j/2+\j/4},{(\i+1/2)*sin(60)}) {};}}
                {\ifthenelse{\j=2 \OR \j=6}{\foreach \i in {1,...,7}{\node[hexa] at ({\j/2+\j/4},{\i*sin(60)}) {};}}
                    {\ifthenelse{\j=3 \OR \j=5}{\foreach \i in {0,...,7}{\node[hexa] at ({\j/2+\j/4},{(\i+1/2)*sin(60)}) {};}}
                        {\foreach \i in {0,...,8}{\node[hexa] at ({\j/2+\j/4},{\i*sin(60)}) {};}}}}}}
    
    \foreach \i in {2,...,6}{
        \draw[fill=black] (0,{(\i+1/2)*sin(60)}) circle (9pt);}
    \foreach \i in {2,...,7}{
        \draw[fill=black] (3/4,{(\i)*sin(60)}) circle (9pt);}
    \foreach \i in {3,...,5}{
        \draw[fill=black] (1.5,{(\i+1/2)*sin(60)}) circle (9pt);}
    \foreach \i in {2,...,6}{
        \draw[fill=lightgray] (6,{(\i+1/2)*sin(60)}) circle (9pt);}
    \foreach \i in {2,...,7}{
        \draw[fill=lightgray] (5.25,{(\i)*sin(60)}) circle (9pt);}
    \foreach \i in {3,...,5}{
        \draw[fill=lightgray] (4.5,{(\i+1/2)*sin(60)}) circle (9pt);}
    
\end{tikzpicture}
}
\end{minipage}
\quad
\begin{minipage}[b]{0.5\linewidth}
\scalebox{0.8}{
\begin{tikzpicture} [hexa/.style= {shape=regular polygon,
                                   regular polygon sides=6,
                                   minimum size=1cm, draw,
                                   inner sep=0,anchor=south,}]
    \foreach \j in {0,...,8}{%
        \ifthenelse{\j=0 \OR \j=8}{\foreach \i in {2,...,6}{\node[hexa] at ({\j/2+\j/4},{\i*sin(60)}) {};}}
            {\ifthenelse{\j=1 \OR \j=7}{\foreach \i in {1,...,6}{\node[hexa] at ({\j/2+\j/4},{(\i+1/2)*sin(60)}) {};}}
                {\ifthenelse{\j=2 \OR \j=6}{\foreach \i in {1,...,7}{\node[hexa] at ({\j/2+\j/4},{\i*sin(60)}) {};}}
                    {\ifthenelse{\j=3 \OR \j=5}{\foreach \i in {0,...,7}{\node[hexa] at ({\j/2+\j/4},{(\i+1/2)*sin(60)}) {};}}
                        {\foreach \i in {0,...,8}{\node[hexa] at ({\j/2+\j/4},{\i*sin(60)}) {};}}}}}}
    
    \foreach \i in {2,3}{
        \draw[fill=lightgray] (0,{(\i+1/2)*sin(60)}) circle (9pt);}
    \foreach \i in {5,6}{
        \draw[fill=black] (0,{(\i+1/2)*sin(60)}) circle (9pt);}
    \foreach \i in {2,3,4}{
        \draw[fill=lightgray] (3/4,{(\i)*sin(60)}) circle (9pt);}
    \foreach \i in {5,6,7}{
        \draw[fill=black] (3/4,{(\i)*sin(60)}) circle (9pt);}
    \foreach \i in {2,3}{
        \draw[fill=lightgray] (1.5,{(\i+1/2)*sin(60)}) circle (9pt);}
    \foreach \i in {5,6}{
        \draw[fill=black] (1.5,{(\i+1/2)*sin(60)}) circle (9pt);}
    
    \foreach \i in {5,6}{
        \draw[fill=lightgray] (6,{(\i+1/2)*sin(60)}) circle (9pt);}
    \foreach \i in {2,3}{
        \draw[fill=black] (6,{(\i+1/2)*sin(60)}) circle (9pt);}
    \foreach \i in {5,6,7}{
        \draw[fill=lightgray] (5.25,{(\i)*sin(60)}) circle (9pt);}
    \foreach \i in {2,3,4}{
        \draw[fill=black] (5.25,{(\i)*sin(60)}) circle (9pt);}
    \foreach \i in {5,6}{
        \draw[fill=lightgray] (4.5,{(\i+1/2)*sin(60)}) circle (9pt);}
    \foreach \i in {2,3}{
        \draw[fill=black] (4.5,{(\i+1/2)*sin(60)}) circle (9pt);}
    
\end{tikzpicture}
}
\end{minipage}
\end{figure}

One well-known strategy for Abalone is to keep one's marbles gathered together into a compact shape because it's easier to attack and defend. Further, it is best to keep these groups of pieces near the middle of the board where they are less likely to be pushed off.  These principles showed up repeatedly in the score functions of Abalone-playing AI agents \cite{aichholzer2002algorithmic, campos, chorus, Lemons, ozcan2004simple, papa, verloop}. They also contributed to strong play in the simplified versions of Abalone that we will explore.  In many cases, occupying the middle of the board proved necessary to win.

A fork serves as an example of an advantage provided by compact configurations. Similar to a fork in chess \cite{lasker2014chess}, a fork threatens two of the opponent's pieces so that the opponent may only save one of them. Forks in Abalone can take many forms. For an example of a fork in $2\times2\times3$ Abalone, see Figure~\ref{fig:forkfun} where a triangular cluster of black pieces threatens two gray pieces, assuring that one of them can be ejected off the board.

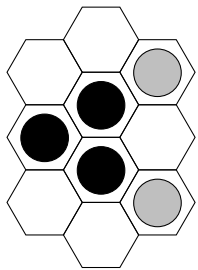
\begin{figure}[ht!]
\caption{Fork example}
\label{fig:forkfun}
\centering
\vspace{.5cm}
    \begin{tikzpicture} [hexa/.style= {shape=regular polygon,
                                   regular polygon sides=6,
                                   minimum size=1cm, draw,
                                   inner sep=0,anchor=south,}]
    \foreach \j in {0,...,2}{%
     \ifodd\j 
         \foreach \i in {0,...,3}{\node[hexa] at ({\j/2+\j/4},{(\i+1/2)*sin(60)}) {};}        
    \else
         \foreach \i in {1,...,3}{\node[hexa] at ({\j/2+\j/4},{\i*sin(60)}) {};}
    \fi}
    \draw[fill=black] (0,{(2+1/2)*sin(60)}) circle (9pt);
    \draw[fill=black] (.75,{(2)*sin(60)}) circle (9pt);
    \draw[fill=black] (.75,{(3)*sin(60)}) circle (9pt);
    \draw[fill=lightgray] (1.5,{(3+1/2)*sin(60)}) circle (9pt);
    \draw[fill=lightgray] (1.5,{(1+1/2)*sin(60)}) circle (9pt);
\end{tikzpicture}
\end{figure}

\subsection{Combinatorial game theory}
\label{subsec:cgt}

To analyze Abalone mathematically, we will use the language and notation of combinatorial game theory as presented in \cite{Siegel}.

A \emph{constellation} is a game board configuration.  That is, each constellation can be represented as a picture of the board where each space contains a black marble, a gray marble, or is empty.  We say that constellations $C$ and $C'$ are \emph{isomorphic} if $C'$ can be obtained from $C$ through a sequence of $60^\circ$ rotations and reflections over the board's 6 axes of symmetry.  The \emph{negative} $-C$ of a constellation $C$ is obtained by swapping the colors of the black and gray marbles in $C$.  We say a board $C$ is \emph{self-negative} if and only if $C$ is isomorphic to $-C$.

The \emph{outcome class} of a constellation $C$, denoted $o(C)$, tells us who will win (if anyone) when gameplay begins at $C$ and both Black and Gray play perfectly.  In this paper, all constellations will belong to one of the 6 following outcome classes:
\begin{itemize}
    \item $\LL$ is the class of constellations which Black (Left) will always win.
    \item $\R$ is the class of constellations which Gray (Right) will always win.
    \item $\D$ is the class of constellations which will always end in a draw.
    \item $\N$ is the class of constellations which will be won by whomever plays next.
    \item $\Nhat$ is the class of constellations which will result in a win for Black (Left) if they play next and a draw if Gray (Right) plays next.
    \item $\Ncheck$ is the class of constellations which will result in a win for Gray (Right) if they play next and a draw if Black (Left) plays next.
\end{itemize}
Notice that if $C$ and $C'$ are isomorphic, then $o(C)=o(C')$ is clearly true.  Further, notice that if $o(C)\in\{\N,\D\}$, then $o(C)=o(-C)$.  This is because the players will simply swap optimal strategies from $C$ to $-C$, making the outcomes identical.  A similar strategy-swapping argument can be used to see that $o(C) = \LL$ if and only if $o(-C)=\R$, and $o(C)=\Nhat$ if and only if $o(-C)=\Ncheck$.

When we say that a game is \emph{weakly solved}, we mean that the outcome class of its official starting constellation is known.  A game is \emph{strongly solved} if the outcome class of every constellation is known.  In both cases, an optimal sequence of moves from the starting constellation is provided.  To provide these sequences of moves, we will refer to the players' options.  When a player can make a move to transition from constellation $C$ to constellation $C'$, we call $C'$ an \emph{option} of that player.  We will denote Black (Left) and Gray (Right)'s sets of options from constellation $C$ as $C^L$ and $C^R$, respectively.  To show both players' options at once, we will write
\[
    C = \{ C^L ~|~ C^R\}.
\]

\subsection{Outline of paper}
\label{subsec:outline}

In Section~\ref{sec:2x2x2}, we describe the $2\times2\times2$ Abalone variant and present a strong solution for it.  In Section~\ref{sec:2x2x3}, we describe the $2\times2\times3$ Abalone variant and determine the outcome class of its starting constellation.  Finally, in Section~\ref{sec:futurework}, we present conjectures and open questions regarding solutions for additional small-board Abalone variants.

\section{A strong solution for 2x2x2 Abalone}
\label{sec:2x2x2}

The first Abalone variant that we will discuss is $2\times 2\times 2$ Abalone.  This variant has the same basic rules and movements as standard Abalone, but it is played on a $2\times 2\times 2$ hexagonal grid, rather than the full-size $5\times 5\times 5$ grid.  Each player starts with two marbles and must push off one of the opponent's marbles to win.  The board begins in the $B1$ configuration depicted in Figure \ref{fig:2x2x2boards}. This configuration is a simplified version of the Belgian daisy configuration from Figure~\ref{fig:5x5x5board}. As mentioned earlier, the Daisy configurations in $5\times 5\times 5$ Abalone were designed to encourage more aggressive play. We use one here to balance the game, as the $2\times2\times2$ analog of the standard setup, labeled as $B0$ in Figure \ref{fig:2x2x2boards}, is an easy win for whomever plays first.

\begin{figure}[p!]
\caption{14 of the 23 nonisomorphic $2\times2\times2$ Abalone constellations.  The 9 nonisomorphic boards which are not pictured are the negatives of boards $B5-B13$.}
\label{fig:2x2x2boards}
\centering
\vspace{.5cm}
    \begin{tikzpicture} [hexa/.style= {shape=regular polygon,
                                   regular polygon sides=6,
                                   minimum size=1cm, draw,
                                   inner sep=0,anchor=south,}]
    \foreach \j in {0,...,2}{%
     \ifodd\j 
         \foreach \i in {0,1,2}{\node[hexa] at ({\j/2+\j/4},{(\i+1/2)*sin(60)}) {};}        
    \else
         \foreach \i in {1,2}{\node[hexa] at ({\j/2+\j/4},{\i*sin(60)}) {};}
    \fi}
    \draw[fill=black] (0,{(1+1/2)*sin(60)}) circle (9pt);
    \draw[fill=black] (0,{(2+1/2)*sin(60)}) circle (9pt);
    \draw[fill=lightgray] (1.5,{(1+1/2)*sin(60)}) circle (9pt);
    \draw[fill=lightgray] (1.5,{(2+1/2)*sin(60)}) circle (9pt);
    \node at (.75,0) {Board $B0$};
\end{tikzpicture}\hspace{.5cm}
\begin{tikzpicture} [hexa/.style= {shape=regular polygon,
                                   regular polygon sides=6,
                                   minimum size=1cm, draw,
                                   inner sep=0,anchor=south,}]

    \foreach \j in {0,...,2}{%
     \ifodd\j 
         \foreach \i in {0,1,2}{\node[hexa] at ({\j/2+\j/4},{(\i+1/2)*sin(60)}) {};}        
    \else
         \foreach \i in {1,2}{\node[hexa] at ({\j/2+\j/4},{\i*sin(60)}) {};}
    \fi}
    \draw[fill=lightgray] (0,{(1+1/2)*sin(60)}) circle (9pt);
    \draw[fill=black] (0,{(2+1/2)*sin(60)}) circle (9pt);
    \draw[fill=black] (1.5,{(1+1/2)*sin(60)}) circle (9pt);
    \draw[fill=lightgray] (1.5,{(2+1/2)*sin(60)}) circle (9pt);
    \node at (.75,0) {Board $B1$};
\end{tikzpicture}\hspace{.5cm}
\begin{tikzpicture} [hexa/.style= {shape=regular polygon,
                                   regular polygon sides=6,
                                   minimum size=1cm, draw,
                                   inner sep=0,anchor=south,}]

    \foreach \j in {0,...,2}{%
     \ifodd\j 
         \foreach \i in {0,1,2}{\node[hexa] at ({\j/2+\j/4},{(\i+1/2)*sin(60)}) {};}        
    \else
         \foreach \i in {1,2}{\node[hexa] at ({\j/2+\j/4},{\i*sin(60)}) {};}
    \fi}
    \draw[fill=lightgray] (0,{(1+1/2)*sin(60)}) circle (9pt);
    \draw[fill=black] (0,{(2+1/2)*sin(60)}) circle (9pt);
    \draw[fill=lightgray] (1.5,{(1+1/2)*sin(60)}) circle (9pt);
    \draw[fill=black] (1.5,{(2+1/2)*sin(60)}) circle (9pt);
    \node at (.75,0) {Board $B2$};
\end{tikzpicture}\hspace{.5cm}
\begin{tikzpicture} [hexa/.style= {shape=regular polygon,
                                   regular polygon sides=6,
                                   minimum size=1cm, draw,
                                   inner sep=0,anchor=south,}]

    \foreach \j in {0,...,2}{%
     \ifodd\j 
         \foreach \i in {0,1,2}{\node[hexa] at ({\j/2+\j/4},{(\i+1/2)*sin(60)}) {};}        
    \else
         \foreach \i in {1,2}{\node[hexa] at ({\j/2+\j/4},{\i*sin(60)}) {};}
    \fi}
    \draw[fill=lightgray] (0,{(1+1/2)*sin(60)}) circle (9pt);
    \draw[fill=black] (0,{(2+1/2)*sin(60)}) circle (9pt);
    \draw[fill=lightgray] (.75,{3*sin(60)}) circle (9pt);
    \draw[fill=black] (.75,{{sin(60)}}) circle (9pt);
    \node at (.75,0) {Board $B3$};
\end{tikzpicture}

\vspace{.5cm}

\begin{tikzpicture} [hexa/.style= {shape=regular polygon,
                                   regular polygon sides=6,
                                   minimum size=1cm, draw,
                                   inner sep=0,anchor=south,}]

    \foreach \j in {0,...,2}{%
     \ifodd\j 
         \foreach \i in {0,1,2}{\node[hexa] at ({\j/2+\j/4},{(\i+1/2)*sin(60)}) {};}        
    \else
         \foreach \i in {1,2}{\node[hexa] at ({\j/2+\j/4},{\i*sin(60)}) {};}
    \fi}
    \draw[fill=lightgray] (0,{(1+1/2)*sin(60)}) circle (9pt);
    \draw[fill=black] (0,{(2+1/2)*sin(60)}) circle (9pt);
    \draw[fill=black] (.75,{3*sin(60)}) circle (9pt);
    \draw[fill=lightgray] (.75,{{sin(60)}}) circle (9pt);
    \node at (.75,0) {Board $B4$};
\end{tikzpicture}\hspace{.5cm}
\begin{tikzpicture} [hexa/.style= {shape=regular polygon,
                                   regular polygon sides=6,
                                   minimum size=1cm, draw,
                                   inner sep=0,anchor=south,}]

    \foreach \j in {0,...,2}{%
     \ifodd\j 
         \foreach \i in {0,1,2}{\node[hexa] at ({\j/2+\j/4},{(\i+1/2)*sin(60)}) {};}        
    \else
         \foreach \i in {1,2}{\node[hexa] at ({\j/2+\j/4},{\i*sin(60)}) {};}
    \fi}
    \draw[fill=lightgray] (0,{(1+1/2)*sin(60)}) circle (9pt);
    \draw[fill=lightgray] (0,{(2+1/2)*sin(60)}) circle (9pt);
    \draw[fill=black] (.75,{3*sin(60)}) circle (9pt);
    \draw[fill=black] (.75,{{sin(60)}}) circle (9pt);
    \node at (.75,0) {Board $B5$};
\end{tikzpicture}\hspace{.5cm}
\begin{tikzpicture} [hexa/.style= {shape=regular polygon,
                                   regular polygon sides=6,
                                   minimum size=1cm, draw,
                                   inner sep=0,anchor=south,}]

    \foreach \j in {0,...,2}{%
     \ifodd\j 
         \foreach \i in {0,1,2}{\node[hexa] at ({\j/2+\j/4},{(\i+1/2)*sin(60)}) {};}        
    \else
         \foreach \i in {1,2}{\node[hexa] at ({\j/2+\j/4},{\i*sin(60)}) {};}
    \fi}
    \draw[fill=lightgray] (0,{(1+1/2)*sin(60)}) circle (9pt);
    \draw[fill=black] (0,{(2+1/2)*sin(60)}) circle (9pt);
    \draw[fill=black] (1.5,{(1+1/2)*sin(60)}) circle (9pt);
    \draw[fill=lightgray] (.75,{(3)*sin(60)}) circle (9pt);
    \node at (.75,0) {Board $B6$};
\end{tikzpicture}\hspace{.5cm}
\begin{tikzpicture} [hexa/.style= {shape=regular polygon,
                                   regular polygon sides=6,
                                   minimum size=1cm, draw,
                                   inner sep=0,anchor=south,}]

    \foreach \j in {0,...,2}{%
     \ifodd\j 
         \foreach \i in {0,1,2}{\node[hexa] at ({\j/2+\j/4},{(\i+1/2)*sin(60)}) {};}        
    \else
         \foreach \i in {1,2}{\node[hexa] at ({\j/2+\j/4},{\i*sin(60)}) {};}
    \fi}
    \draw[fill=lightgray] (0,{(1+1/2)*sin(60)}) circle (9pt);
    \draw[fill=lightgray] (0,{(2+1/2)*sin(60)}) circle (9pt);
    \draw[fill=black] (1.5,{(1+1/2)*sin(60)}) circle (9pt);
    \draw[fill=black] (.75,{(3)*sin(60)}) circle (9pt);
    \node at (.75,0) {Board $B7$};
\end{tikzpicture}\vspace{.5cm}

\begin{tikzpicture} [hexa/.style= {shape=regular polygon,
                                   regular polygon sides=6,
                                   minimum size=1cm, draw,
                                   inner sep=0,anchor=south,}]

    \foreach \j in {0,...,2}{%
     \ifodd\j 
         \foreach \i in {0,1,2}{\node[hexa] at ({\j/2+\j/4},{(\i+1/2)*sin(60)}) {};}        
    \else
         \foreach \i in {1,2}{\node[hexa] at ({\j/2+\j/4},{\i*sin(60)}) {};}
    \fi}
    \draw[fill=lightgray] (.75,{(1)*sin(60)}) circle (9pt);
    \draw[fill=black] (0,{(2+1/2)*sin(60)}) circle (9pt);
    \draw[fill=lightgray] (1.5,{(1+1/2)*sin(60)}) circle (9pt);
    \draw[fill=black] (.75,{(2)*sin(60)}) circle (9pt);
    \node at (.75,0) {Board $B8$};
\end{tikzpicture}\hspace{.5cm}
\begin{tikzpicture} [hexa/.style= {shape=regular polygon,
                                   regular polygon sides=6,
                                   minimum size=1cm, draw,
                                   inner sep=0,anchor=south,}]

    \foreach \j in {0,...,2}{%
     \ifodd\j 
         \foreach \i in {0,1,2}{\node[hexa] at ({\j/2+\j/4},{(\i+1/2)*sin(60)}) {};}        
    \else
         \foreach \i in {1,2}{\node[hexa] at ({\j/2+\j/4},{\i*sin(60)}) {};}
    \fi}
    \draw[fill=lightgray] (.75,{(1)*sin(60)}) circle (9pt);
    \draw[fill=black] (0,{(2+1/2)*sin(60)}) circle (9pt);
    \draw[fill=lightgray] (1.5,{(2+1/2)*sin(60)}) circle (9pt);
    \draw[fill=black] (.75,{(2)*sin(60)}) circle (9pt);
    \node at (.75,0) {Board $B9$};
\end{tikzpicture}\hspace{.5cm}
\begin{tikzpicture} [hexa/.style= {shape=regular polygon,
                                   regular polygon sides=6,
                                   minimum size=1cm, draw,
                                   inner sep=0,anchor=south,}]

    \foreach \j in {0,...,2}{%
     \ifodd\j 
         \foreach \i in {0,1,2}{\node[hexa] at ({\j/2+\j/4},{(\i+1/2)*sin(60)}) {};}        
    \else
         \foreach \i in {1,2}{\node[hexa] at ({\j/2+\j/4},{\i*sin(60)}) {};}
    \fi}
    \draw[fill=lightgray] (0,{(1+1/2)*sin(60)}) circle (9pt);
    \draw[fill=black] (0,{(2+1/2)*sin(60)}) circle (9pt);
    \draw[fill=lightgray] (1.5,{(2+1/2)*sin(60)}) circle (9pt);
    \draw[fill=black] (.75,{(2)*sin(60)}) circle (9pt);
    \node at (.75,0) {Board $B10$};
\end{tikzpicture}\hspace{.5cm}
\begin{tikzpicture} [hexa/.style= {shape=regular polygon,
                                   regular polygon sides=6,
                                   minimum size=1cm, draw,
                                   inner sep=0,anchor=south,}]

    \foreach \j in {0,...,2}{%
     \ifodd\j 
         \foreach \i in {0,1,2}{\node[hexa] at ({\j/2+\j/4},{(\i+1/2)*sin(60)}) {};}        
    \else
         \foreach \i in {1,2}{\node[hexa] at ({\j/2+\j/4},{\i*sin(60)}) {};}
    \fi}
    \draw[fill=lightgray] (0,{(1+1/2)*sin(60)}) circle (9pt);
    \draw[fill=black] (0,{(2+1/2)*sin(60)}) circle (9pt);
    \draw[fill=lightgray] (.75,{(1)*sin(60)}) circle (9pt);
    \draw[fill=black] (.75,{(2)*sin(60)}) circle (9pt);
    \node at (.75,0) {Board $B11$};
\end{tikzpicture}\vspace{.5cm}

\begin{tikzpicture} [hexa/.style= {shape=regular polygon,
                                   regular polygon sides=6,
                                   minimum size=1cm, draw,
                                   inner sep=0,anchor=south,}]

    \foreach \j in {0,...,2}{%
     \ifodd\j 
         \foreach \i in {0,1,2}{\node[hexa] at ({\j/2+\j/4},{(\i+1/2)*sin(60)}) {};}        
    \else
         \foreach \i in {1,2}{\node[hexa] at ({\j/2+\j/4},{\i*sin(60)}) {};}
    \fi}
    \draw[fill=lightgray] (0,{(1+1/2)*sin(60)}) circle (9pt);
    \draw[fill=black] (0,{(2+1/2)*sin(60)}) circle (9pt);
    \draw[fill=lightgray] (1.5,{(1+1/2)*sin(60)}) circle (9pt);
    \draw[fill=black] (.75,{(2)*sin(60)}) circle (9pt);
    \node at (.75,0) {Board $B12$};
\end{tikzpicture}\hspace{.5cm}
\begin{tikzpicture} [hexa/.style= {shape=regular polygon,
                                   regular polygon sides=6,
                                   minimum size=1cm, draw,
                                   inner sep=0,anchor=south,}]

    \foreach \j in {0,...,2}{%
     \ifodd\j 
         \foreach \i in {0,1,2}{\node[hexa] at ({\j/2+\j/4},{(\i+1/2)*sin(60)}) {};}        
    \else
         \foreach \i in {1,2}{\node[hexa] at ({\j/2+\j/4},{\i*sin(60)}) {};}
    \fi}
    \draw[fill=lightgray] (0,{(1+1/2)*sin(60)}) circle (9pt);
    \draw[fill=black] (0,{(2+1/2)*sin(60)}) circle (9pt);
    \draw[fill=lightgray] (.75,{(3)*sin(60)}) circle (9pt);
    \draw[fill=black] (.75,{(2)*sin(60)}) circle (9pt);
    \node at (.75,0) {Board $B13$};
\end{tikzpicture}\hspace{.5cm}
\end{figure}
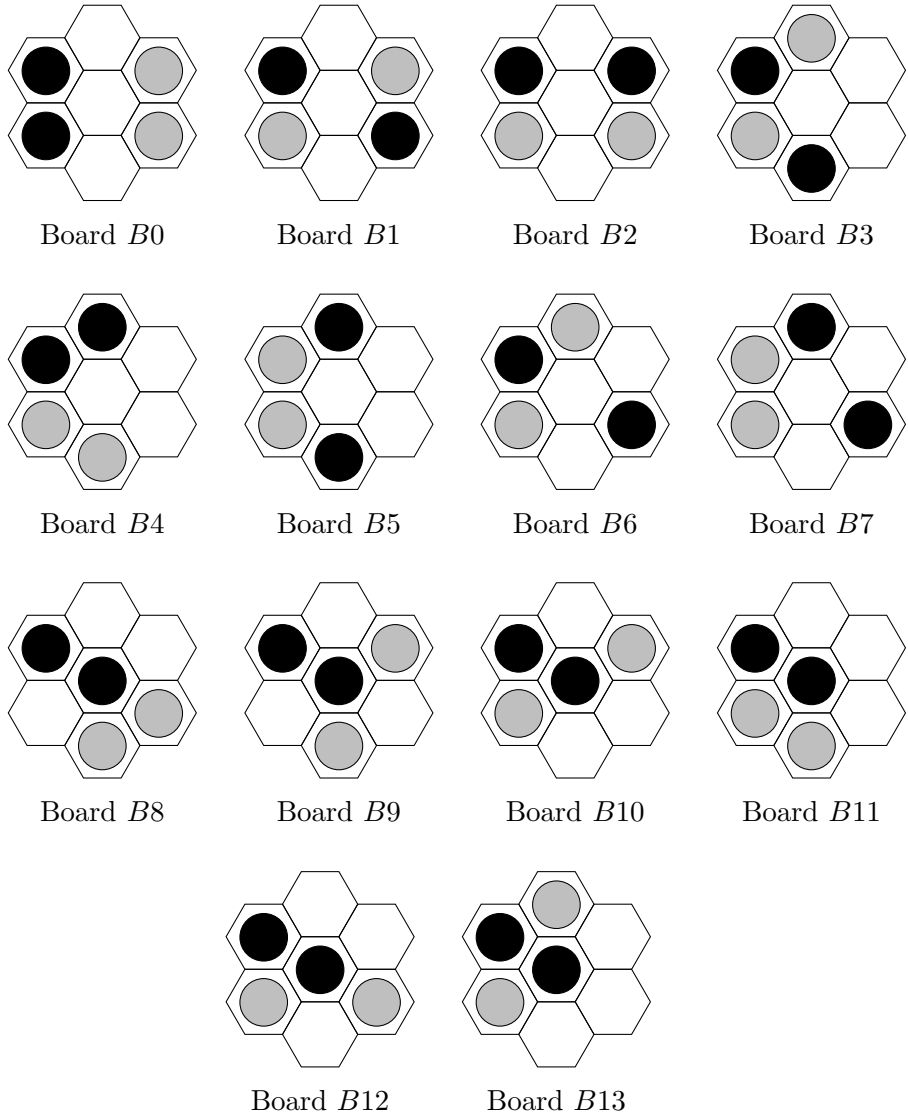

On the $2\times2\times2$ Abalone board, there are seven spaces for the four marbles to occupy.  There are exactly 23 nonisomorphic boards in $2\times 2\times 2$ Abalone. Figure \ref{fig:2x2x2boards} shows 14 of these nonisomorphic constellations.  Notice that boards $B0$-$B4$ are all self-negative and boards $B5$-$B13$ are not.  These nonisomorphic negatives are exactly the 9 nonisomorphic boards that we have excluded from Figure \ref{fig:2x2x2boards} for brevity.

The next two lemmas are easily seen and thus presented without proof.  Notice that Lemma~\ref{lem:negs} is true by a strategy-swapping argument and Lemma~\ref{lem:iso} is obvious by the definition of isomorphism.
\begin{lemma}
\label{lem:negs}
Let $C$ be a constellation.  Then $o(C)=\LL$ if and only if $o(-C)=\R$.  Similarly, $o(C)=\Nhat$ if and only if $o(-C)=\Ncheck$.
\end{lemma}
%
%

\begin{lemma}
\label{lem:iso}
If $C$ and $C'$ are isomorphic constellations, then $o(C)=o(C')$.
\end{lemma}

We are now ready to strongly solve $2\times2\times2$ Abalone.
\begin{theorem}
    $2\times 2\times 2$ Abalone is strongly solved.  An exhaustive list of its constellations' outcome classes is provided in Table~\ref{table:allbds}.
\end{theorem}
\begin{proof}
    As there are 23 nonisomorphic boards possible in $2\times 2\times 2$ Abalone, we will determine the outcome class of all 23 constellations, assuming the players are perfectly rational and omniscient.  In doing so, we will also provide optimal sequences of moves from each starting constellation.
    
    To begin, notice that $B11$ is a clear win for Black. If Gray moves first, they must move to $B12$. This endangers their marble by moving it into the way of Black's two marbles, so Black can respond by pushing Gray off and winning. If Black moves first at $B11$, they can simply move both of their marbles across the middle and maintain an isomorphic board. Thus, $o(B11)=\LL$.  By Lemma \ref{lem:negs}, we may also conclude that $o(-B11)=\R$.
    
    Next, notice that the first player to move in $B0$ can immediately force the game into $B11$ or $-B11$ (whichever benefits them) by moving two of their marbles into the middle column. As a result, we can conclude that $o(B0)=\N$.  By a similar argument, we can also conclude that $o(B4)=\N$.
    
    Based on the outcomes for boards $B0$, $B4$, and $\pm B11$, we can conclude that Black will win if they play first on $B5$, $B7$, $B8$, or $B12$.  Similarly, Gray will win if they play first on $-B5$, $-B7$, $-B8$, or $-B12$. So if both players are rational, Black will never opt to move to $B0$, $B4$, $-B5$, $-B7$, $-B8$, $-B11$, or $-B12$, and Gray will never opt to move to $B0$, $B4$, $B5$, $B7$, $B8$, $B11$, or $B12$.  
    
    Now consider the boards $B1$, $B2$, $B3$, $\pm B6$, $\pm B9$, $\pm B10$, and $\pm B13$.  The Black (Left) and Gray (Right) options for these 11 boards are listed below.  Notice that the irrational moves mentioned earlier have been pruned.  As a result, moves to $\pm B12$ have been been reduced in the following way: If, for example, Black moves to board $B12$, then Gray's options are to move to board $B8$, $B10$, or $B11$.  This means Gray's only logical response is to move to board $B10$.  From $B10$, Black has the option to go to board $B1$, $-B5$, $-B6$, or $B10$.  After pruning the illogical option of moving to $-B5$, we get that an option for Black to move to $B12$ reduces to the option for Black to move to $B1$, $-B6$ or $B10$.  By a similar argument, Gray's option to move to $-B12$ reduces to the option for Gray to move to $B1$, $B6$, or $-B10$.
    \begin{align*}
        B1 &= \{-B6, B10 ~|~ B6, -B10\} \\
        B2 &= \{-B7, B12 ~|~ B7, -B12\} \\
           &= \{B1, -B6, B10 ~|~ B1, B6, -B10\} \\
        B3 &= \{B6, B12, B13 ~|~ -B6, -B12, -B13\} \\
           &= \{B1, \pm B6, B10, B13 ~|~ B1, \pm B6, -B10, -B13\}\\
        B6 &= \{B3, B9, B13 ~|~ B1, -B10\} \\
        -B6 &= \{B1, B10 ~|~ B3, -B9, -B13\} \\
        B9 &= \{B6, -B7, B12, B13 ~|~ B8, B10\} \\
           &= \{B1, \pm B6, B10, B13 ~|~ B10\} \\
        -B9 &= \{-B8, -B10 ~|~ -B6, B7, -B12, -B13\} \\
            &= \{-B10 ~|~ B1, \pm B6, -B10, -B13\} \\
        B10 &= \{B1, -B5, -B6, B10 ~|~ B9, B12, B13\} \\
            &= \{B1, -B6, B10 ~|~ B9, B13\} \\
        -B10 &= \{-B9, -B12, -B13 ~|~ B1, B5, B6, -B10\} \\
             &= \{-B9, -B13 ~|~ B1, B6, -B10\} \\
        B13 &= \{B3, B6, B9 ~|~ B10\} \\
        -B13 &= \{-B10 ~|~ B3, -B6, -B9\}
    \end{align*}
    Based on these options, we may conclude that any optimal sequence of plays starting with a board in the set
    \[ S=\{B1, B2, B3, \pm B6, \pm B9, \pm B10, \pm B13\}\]
    will inevitably return to a board in $S$.  Thus all optimal runs starting in $S$ are infinite, meaning $o(C) = \D$ for all $C\in S$.
    
    Finally, let's consider the boards $\pm B5$, $\pm B7$, $\pm B8$, and $\pm B12$.  Recall that Black can win if they play first on $B5$, $B7$, $B8$, and $B12$.  The Gray (Right) options for each of these boards are listed below.  Notice that these options lists have also been reduced and pruned to avoid illogical moves.
    \begin{align*}
        B5^R &= \{ -B10\} \\
        B7^R &= \{B2, -B9, -B12 \} = \{B1, B2, B6, -B9, -B10\} \\
        B8^R &= \{B9, B12\} = \{B9\} \\
        B12^R &= \{B8, B10, B11\} = \{B10\}
    \end{align*}
    Since all right options are in outcome class $\D$, \[o(B5) = o(B7) = o(B8) = o(B12) = \Nhat,\] and by Lemma \ref{lem:negs}, \[o(-B5) = o(-B7)= o(-B8) = o(-B12) = \Ncheck.\] 
    
    We have now determined the outcome class of all 23 nonisomorphic $2\times 2\times 2$ Abalone boards.  By Lemma \ref{lem:iso}, we therefore know the outcome class of every possible constellation and the game is strongly solved.  For convenience, all outcomes are summarized in Table \ref{table:allbds}.
\end{proof}

    \begin{table}[ht!]
    \centering
    \caption{Outcomes for all nonisomorphic $2\times 2\times 2$ Abalone boards.}
    \label{table:allbds}
    {\renewcommand{\arraystretch}{1.5}
    \begin{tabular}{|c|c|c|c|c|c|} 
    \hline
     & $B0$ & $B1$ & $B2$ & $B3$ & $B4$ \\ 
    \hline
    $o(C)$ & $\N$ & $\D$ & $\D$ & $\D$ & $\N$ \\ 
    \hline
    \end{tabular}}
    
    \vspace{.5cm}
    
    {\renewcommand{\arraystretch}{1.5}
    \begin{tabular}{|c|c|c|c|c|c|c|c|c|c|c|c|c|c|} 
    \hline
     & $B5$ & $B6$ & $B7$ & $B8$ & $B9$ & $B10$ & $B11$ & $B12$ & $B13$ \\ 
    \hline
    $o(C)$ & $\Nhat$ & $\D$ & $\Nhat$ & $\Nhat$ & $\D$ & $\D$ & $\LL$ & $\Nhat$ & $\D$ \\ 
    \hline
    $o(-C)$ & $\Ncheck$ & $\D$ & $\Ncheck$ & $\Ncheck$ & $\D$ & $\D$ & $\R$ & $\Ncheck$ & $\D$ \\ 
    \hline
    \end{tabular}}
    \end{table}

\section{A weak solution for 2x2x3 Abalone}
\label{sec:2x2x3}
{$2\times 2\times 3$} Abalone is a variant of Abalone that is played on a non-regular $2\times2\times3$ hexagonal grid with 10 spaces. Each player begins with 3 marbles and must push off one of their opponent's marbles to win. The starting configuration is a simplified version of the standard setup and is labeled $C0$ in Figure~\ref{fig:2x2x3optimalplayboards}. Note that the irregular board shape and the fact that each player has an odd number of pieces prevents there from being a Daisy-type constellation on this board.

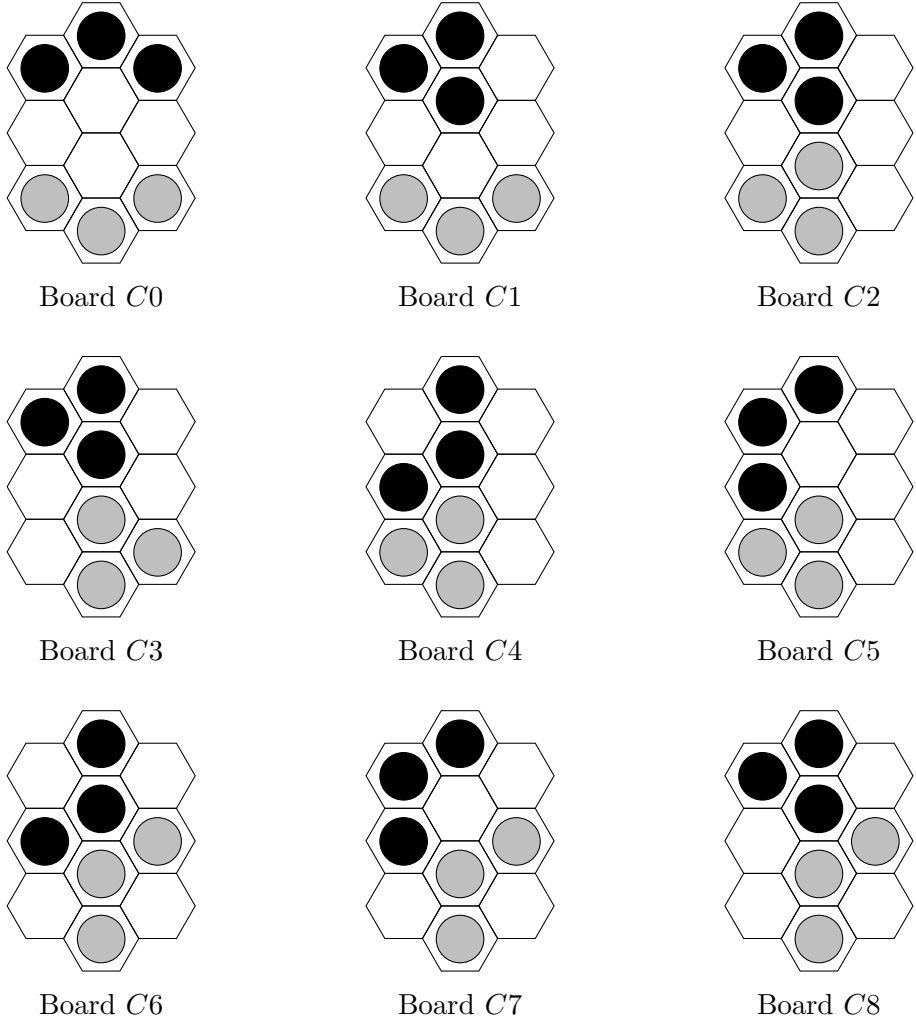
\begin{figure}[ht!]
\caption{Constellations that may occur during optimal $2\times 2\times 3$ Abalone play.}
\label{fig:2x2x3optimalplayboards}
\centering
\vspace{.5cm}
    \begin{tikzpicture} [hexa/.style= {shape=regular polygon,
                                   regular polygon sides=6,
                                   minimum size=1cm, draw,
                                   inner sep=0,anchor=south,}]
    \foreach \j in {0,...,2}{%
     \ifodd\j 
         \foreach \i in {0,...,3}{\node[hexa] at ({\j/2+\j/4},{(\i+1/2)*sin(60)}) {};}        
    \else
         \foreach \i in {1,...,3}{\node[hexa] at ({\j/2+\j/4},{\i*sin(60)}) {};}
    \fi}
    \draw[fill=black] (0,{(3+1/2)*sin(60)}) circle (9pt);
    \draw[fill=black] (.75,{(4)*sin(60)}) circle (9pt);
    \draw[fill=black] (1.5,{(3.5)*sin(60)}) circle (9pt);
    \draw[fill=lightgray] (0,{(1+1/2)*sin(60)}) circle (9pt);
    \draw[fill=lightgray] (.75,{(1)*sin(60)}) circle (9pt);
    \draw[fill=lightgray] (1.5,{(1+1/2)*sin(60)}) circle (9pt);
    \node at (.75,0) {Board $C0$};
\end{tikzpicture}
\hspace{2cm}
    \begin{tikzpicture} [hexa/.style= {shape=regular polygon,
                                   regular polygon sides=6,
                                   minimum size=1cm, draw,
                                   inner sep=0,anchor=south,}]
    \foreach \j in {0,...,2}{%
     \ifodd\j 
         \foreach \i in {0,...,3}{\node[hexa] at ({\j/2+\j/4},{(\i+1/2)*sin(60)}) {};}        
    \else
         \foreach \i in {1,...,3}{\node[hexa] at ({\j/2+\j/4},{\i*sin(60)}) {};}
    \fi}
    \draw[fill=black] (0,{(3+1/2)*sin(60)}) circle (9pt);
    \draw[fill=black] (.75,{(3)*sin(60)}) circle (9pt);
    \draw[fill=black] (.75,{(4)*sin(60)}) circle (9pt);
    \draw[fill=lightgray] (0,{(1+1/2)*sin(60)}) circle (9pt);
    \draw[fill=lightgray] (.75,{(1)*sin(60)}) circle (9pt);
    \draw[fill=lightgray] (1.5,{(1+1/2)*sin(60)}) circle (9pt);
    \node at (.75,0) {Board $C1$};
\end{tikzpicture}
\hspace{2cm}
    \begin{tikzpicture} [hexa/.style= {shape=regular polygon,
                                   regular polygon sides=6,
                                   minimum size=1cm, draw,
                                   inner sep=0,anchor=south,}]
    \foreach \j in {0,...,2}{%
     \ifodd\j 
         \foreach \i in {0,...,3}{\node[hexa] at ({\j/2+\j/4},{(\i+1/2)*sin(60)}) {};}        
    \else
         \foreach \i in {1,...,3}{\node[hexa] at ({\j/2+\j/4},{\i*sin(60)}) {};}
    \fi}
    \draw[fill=black] (0,{(3+1/2)*sin(60)}) circle (9pt);
    \draw[fill=black] (.75,{(3)*sin(60)}) circle (9pt);
    \draw[fill=black] (.75,{(4)*sin(60)}) circle (9pt);
    \draw[fill=lightgray] (0,{(1+1/2)*sin(60)}) circle (9pt);
    \draw[fill=lightgray] (.75,{(1)*sin(60)}) circle (9pt);
    \draw[fill=lightgray] (.75,{(2)*sin(60)}) circle (9pt);
    \node at (.75,0) {Board $C2$};
\end{tikzpicture}\\
\vspace{.5cm}
    \begin{tikzpicture} [hexa/.style= {shape=regular polygon,
                                   regular polygon sides=6,
                                   minimum size=1cm, draw,
                                   inner sep=0,anchor=south,}]
    \foreach \j in {0,...,2}{%
     \ifodd\j 
         \foreach \i in {0,...,3}{\node[hexa] at ({\j/2+\j/4},{(\i+1/2)*sin(60)}) {};}        
    \else
         \foreach \i in {1,...,3}{\node[hexa] at ({\j/2+\j/4},{\i*sin(60)}) {};}
    \fi}
    \draw[fill=black] (0,{(3+1/2)*sin(60)}) circle (9pt);
    \draw[fill=black] (.75,{(3)*sin(60)}) circle (9pt);
    \draw[fill=black] (.75,{(4)*sin(60)}) circle (9pt);
    \draw[fill=lightgray] (.75,{(2)*sin(60)}) circle (9pt);
    \draw[fill=lightgray] (.75,{(1)*sin(60)}) circle (9pt);
    \draw[fill=lightgray] (1.5,{(1+1/2)*sin(60)}) circle (9pt);
    \node at (.75,0) {Board $C3$};
\end{tikzpicture}
\hspace{2cm}
    \begin{tikzpicture} [hexa/.style= {shape=regular polygon,
                                   regular polygon sides=6,
                                   minimum size=1cm, draw,
                                   inner sep=0,anchor=south,}]
    \foreach \j in {0,...,2}{%
     \ifodd\j 
         \foreach \i in {0,...,3}{\node[hexa] at ({\j/2+\j/4},{(\i+1/2)*sin(60)}) {};}        
    \else
         \foreach \i in {1,...,3}{\node[hexa] at ({\j/2+\j/4},{\i*sin(60)}) {};}
    \fi}
    \draw[fill=black] (0,{(2+1/2)*sin(60)}) circle (9pt);
    \draw[fill=black] (.75,{(3)*sin(60)}) circle (9pt);
    \draw[fill=black] (.75,{(4)*sin(60)}) circle (9pt);
    \draw[fill=lightgray] (0,{(1+1/2)*sin(60)}) circle (9pt);
    \draw[fill=lightgray] (.75,{(1)*sin(60)}) circle (9pt);
    \draw[fill=lightgray] (.75,{(2)*sin(60)}) circle (9pt);
    \node at (.75,0) {Board $C4$};
\end{tikzpicture}
\hspace{2cm}
    \begin{tikzpicture} [hexa/.style= {shape=regular polygon,
                                   regular polygon sides=6,
                                   minimum size=1cm, draw,
                                   inner sep=0,anchor=south,}]
    \foreach \j in {0,...,2}{%
     \ifodd\j 
         \foreach \i in {0,...,3}{\node[hexa] at ({\j/2+\j/4},{(\i+1/2)*sin(60)}) {};}        
    \else
         \foreach \i in {1,...,3}{\node[hexa] at ({\j/2+\j/4},{\i*sin(60)}) {};}
    \fi}
    \draw[fill=black] (0,{(3+1/2)*sin(60)}) circle (9pt);
    \draw[fill=black] (0,{(2.5)*sin(60)}) circle (9pt);
    \draw[fill=black] (.75,{(4)*sin(60)}) circle (9pt);
    \draw[fill=lightgray] (0,{(1+1/2)*sin(60)}) circle (9pt);
    \draw[fill=lightgray] (.75,{(1)*sin(60)}) circle (9pt);
    \draw[fill=lightgray] (.75,{(2)*sin(60)}) circle (9pt);
    \node at (.75,0) {Board $C5$};
\end{tikzpicture}\\
\vspace{.5cm}
    \begin{tikzpicture} [hexa/.style= {shape=regular polygon,
                                   regular polygon sides=6,
                                   minimum size=1cm, draw,
                                   inner sep=0,anchor=south,}]
    \foreach \j in {0,...,2}{%
     \ifodd\j 
         \foreach \i in {0,...,3}{\node[hexa] at ({\j/2+\j/4},{(\i+1/2)*sin(60)}) {};}        
    \else
         \foreach \i in {1,...,3}{\node[hexa] at ({\j/2+\j/4},{\i*sin(60)}) {};}
    \fi}
    \draw[fill=black] (0,{(2+1/2)*sin(60)}) circle (9pt);
    \draw[fill=black] (.75,{(3)*sin(60)}) circle (9pt);
    \draw[fill=black] (.75,{(4)*sin(60)}) circle (9pt);
    \draw[fill=lightgray] (.75,{(2)*sin(60)}) circle (9pt);
    \draw[fill=lightgray] (.75,{(1)*sin(60)}) circle (9pt);
    \draw[fill=lightgray] (1.5,{(2+1/2)*sin(60)}) circle (9pt);
    \node at (.75,0) {Board $C6$};
\end{tikzpicture}
\hspace{2cm}
    \begin{tikzpicture} [hexa/.style= {shape=regular polygon,
                                   regular polygon sides=6,
                                   minimum size=1cm, draw,
                                   inner sep=0,anchor=south,}]
    \foreach \j in {0,...,2}{%
     \ifodd\j 
         \foreach \i in {0,...,3}{\node[hexa] at ({\j/2+\j/4},{(\i+1/2)*sin(60)}) {};}        
    \else
         \foreach \i in {1,...,3}{\node[hexa] at ({\j/2+\j/4},{\i*sin(60)}) {};}
    \fi}
    \draw[fill=black] (0,{(2+1/2)*sin(60)}) circle (9pt);
    \draw[fill=black] (0,{(3.5)*sin(60)}) circle (9pt);
    \draw[fill=black] (.75,{(4)*sin(60)}) circle (9pt);
    \draw[fill=lightgray] (1.5,{(2+1/2)*sin(60)}) circle (9pt);
    \draw[fill=lightgray] (.75,{(1)*sin(60)}) circle (9pt);
    \draw[fill=lightgray] (.75,{(2)*sin(60)}) circle (9pt);
    \node at (.75,0) {Board $C7$};
\end{tikzpicture}
\hspace{2cm}
    \begin{tikzpicture} [hexa/.style= {shape=regular polygon,
                                   regular polygon sides=6,
                                   minimum size=1cm, draw,
                                   inner sep=0,anchor=south,}]
    \foreach \j in {0,...,2}{%
     \ifodd\j 
         \foreach \i in {0,...,3}{\node[hexa] at ({\j/2+\j/4},{(\i+1/2)*sin(60)}) {};}        
    \else
         \foreach \i in {1,...,3}{\node[hexa] at ({\j/2+\j/4},{\i*sin(60)}) {};}
    \fi}
    \draw[fill=black] (0,{(3+1/2)*sin(60)}) circle (9pt);
    \draw[fill=black] (.75,{(3)*sin(60)}) circle (9pt);
    \draw[fill=black] (.75,{(4)*sin(60)}) circle (9pt);
    \draw[fill=lightgray] (1.5,{(2+1/2)*sin(60)}) circle (9pt);
    \draw[fill=lightgray] (.75,{(1)*sin(60)}) circle (9pt);
    \draw[fill=lightgray] (.75,{(2)*sin(60)}) circle (9pt);
    \node at (.75,0) {Board $C8$};
\end{tikzpicture}
\end{figure}

Using the SageMath computer algebra system, we wrote a program which found over 555 nonisomorphic $2\times2\times3$ Abalone boards.  This large number of constellations makes it arduous to strongly solve $2\times2\times3$ Abalone by hand, so we instead focus on a weak solution.  In doing so, we will show that there are 13 non-isomorphic boards that may occur during an optimal game of $2\times 2\times 3$ Abalone.  Of these boards, 9 are depicted in Figure~\ref{fig:2x2x3optimalplayboards}.  The other 4 are $-C4$, $-C5$, $-C7$, and $-C8$. 

We begin by proving Lemma~\ref{lem:223middle}, which determines the outcome classes of boards in which a single player controls both middle positions.  In particular, it shows that a player who is able to create a central triangle subconfiguration is able to force a win.  Thus, during optimal play, both players must prevent their opponent from creating this triangle formation.

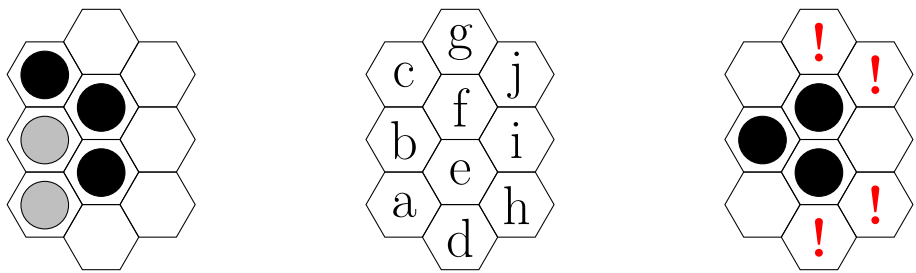
\begin{figure}[ht!]
\caption{A neutral subconfiguration (left), a labeled guide to reference spaces (center), and a Black triangle subconfiguration (right). }
\label{fig:2x2x3neutral}
\centering
\vspace{.5cm}
    \begin{tikzpicture} [hexa/.style= {shape=regular polygon,
                                   regular polygon sides=6,
                                   minimum size=1cm, draw,
                                   inner sep=0,anchor=south,}]
    \foreach \j in {0,...,2}{%
     \ifodd\j 
         \foreach \i in {0,...,3}{\node[hexa] at ({\j/2+\j/4},{(\i+1/2)*sin(60)}) {};}        
    \else
         \foreach \i in {1,...,3}{\node[hexa] at ({\j/2+\j/4},{\i*sin(60)}) {};}
    \fi}
    \draw[fill=black] (0,{(3+1/2)*sin(60)}) circle (9pt);
    \draw[fill=black] (.75,{(2)*sin(60)}) circle (9pt);
    \draw[fill=black] (.75,{(3)*sin(60)}) circle (9pt);
    \draw[fill=lightgray] (0,{(1+1/2)*sin(60)}) circle (9pt);
    \draw[fill=lightgray] (0,{(2+1/2)*sin(60)}) circle (9pt);
\end{tikzpicture}
\hspace{2cm}
    \begin{tikzpicture} [hexa/.style= {shape=regular polygon,
                                   regular polygon sides=6,
                                   minimum size=1cm, draw,
                                   inner sep=0,anchor=south,}]
    \foreach \j in {0,...,2}{%
     \ifodd\j 
         \foreach \i in {0,...,3}{\node[hexa] at ({\j/2+\j/4},{(\i+1/2)*sin(60)}) {};}        
    \else
         \foreach \i in {1,...,3}{\node[hexa] at ({\j/2+\j/4},{\i*sin(60)}) {};}
    \fi}
    \node at (0,{(3+1/2)*sin(60)}) {\huge c};
    \node at (0,{(2+1/2)*sin(60)}) {\huge b};
    \node at (0,{(1+1/2)*sin(60)}) {\huge a};
    \node at (.75,{(4)*sin(60)}) {\huge g};
    \node at (.75,{(3)*sin(60)}) {\huge f};
    \node at (.75,{(2)*sin(60)}) {\huge e};
    \node at (.75,{(1)*sin(60)}) {\huge d};
    \node at (1.5,{(1+1/2)*sin(60)}) {\huge h};
    \node at (1.5,{(2+1/2)*sin(60)}) {\huge i};
    \node at (1.5,{(3+1/2)*sin(60)}) {\huge j};
\end{tikzpicture}
\hspace{2cm}
\begin{tikzpicture} [hexa/.style= {shape=regular polygon,
                                   regular polygon sides=6,
                                   minimum size=1cm, draw,
                                   inner sep=0,anchor=south,}]
    \foreach \j in {0,...,2}{%
     \ifodd\j 
         \foreach \i in {0,...,3}{\node[hexa] at ({\j/2+\j/4},{(\i+1/2)*sin(60)}) {};}        
    \else
         \foreach \i in {1,...,3}{\node[hexa] at ({\j/2+\j/4},{\i*sin(60)}) {};}
    \fi}
    \draw[fill=black] (0,{(2+1/2)*sin(60)}) circle (9pt);
    \draw[fill=black] (.75,{(2)*sin(60)}) circle (9pt);
    \draw[fill=black] (.75,{(3)*sin(60)}) circle (9pt);
    
    \node at (1.5,{(3+1/2)*sin(60)}) {{\color{red} \huge\textbf{!}}};
    \node at (.75,{(1)*sin(60)}) {{\color{red} \huge\textbf{!}}};
    \node at (1.5,{(1+1/2)*sin(60)}) {{\color{red} \huge\textbf{!}}};
    \node at (.75,{(4)*sin(60)}) {{\color{red} \huge\textbf{!}}};
\end{tikzpicture}
\end{figure}

\begin{lemma}
    Let $C$ be a $2\times2\times3$ Abalone constellation in which Black occupies both of the middle two spaces (denoted as spaces $e$ and $f$ in Figure~\ref{fig:2x2x3neutral}). Then either $o(C)=\LL$, or $C$ contains the neutral subconfiguration shown in Figure~\ref{fig:2x2x3neutral}. In the latter case, $o(C)=\N$.
    \label{lem:223middle}
\end{lemma}

\begin{proof}
Suppose $C$ is a $2\times2\times3$ Abalone constellation in which Black occupies both of the middle two spaces ($e$ and $f$).  Notice that neither of these middle marbles can be pushed out of place, because there isn't sufficient room to create a Gray line that threatens them.

We'll begin by discussing the strong triangular formation shown in Figure \ref{fig:2x2x3neutral} where the Black marbles are in spaces $e$, $f$, and either $b$ or $i$.  In this formation, Black threatens all but 3 positions on the board; the threatened spaces are labeled with exclamation points in the figure.

Since Black occupies the two middle spaces, they can almost always get to this triangle formation from $C$ in three moves or fewer.  (There is one exception, which we will address later in the proof.) If their third marble is already in position $b$ or $i$, then no moves are needed.  If their third marble is in position $a$, $c$, $h$, or $j$, the triangle formation will only take one move.  They can move their third marble into an adjacent empty $b$ or $i$ space, or they can use inline movement to create the triangle as illustrated in Figure~\ref{fig:2x2x3secure}.  If necessary, this inline movement can be a sumito move, which would end the game.

{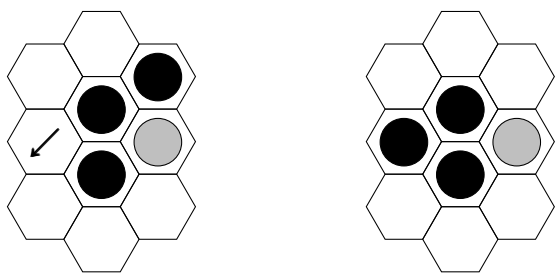
\begin{figure}[ht!]
\caption{Inline movement to create a triangle in one move.}
\label{fig:2x2x3secure}
\centering
\vspace{.5cm}
    \begin{tikzpicture} [hexa/.style= {shape=regular polygon,
                                   regular polygon sides=6,
                                   minimum size=1cm, draw,
                                   inner sep=0,anchor=south,}]
    \foreach \j in {0,...,2}{%
     \ifodd\j 
         \foreach \i in {0,...,3}{\node[hexa] at ({\j/2+\j/4},{(\i+1/2)*sin(60)}) {};}        
    \else
         \foreach \i in {1,...,3}{\node[hexa] at ({\j/2+\j/4},{\i*sin(60)}) {};}
    \fi}
    \draw[fill=black] (1.5,{(3+1/2)*sin(60)}) circle (9pt);
    \draw[fill=black] (.75,{(2)*sin(60)}) circle (9pt);
    \draw[fill=black] (.75,{(3)*sin(60)}) circle (9pt);

    \draw[fill=lightgray] (1.5,{(2+1/2)*sin(60)}) circle (9pt);
    
    \node at (0,{(2+1/2)*sin(60)}) {\large $\boldsymbol\swarrow$};
    
\end{tikzpicture}
\hspace{2cm}
     \begin{tikzpicture} [hexa/.style= {shape=regular polygon,
                                   regular polygon sides=6,
                                   minimum size=1cm, draw,
                                   inner sep=0,anchor=south,}]
    \foreach \j in {0,...,2}{%
     \ifodd\j 
         \foreach \i in {0,...,3}{\node[hexa] at ({\j/2+\j/4},{(\i+1/2)*sin(60)}) {};}        
    \else
         \foreach \i in {1,...,3}{\node[hexa] at ({\j/2+\j/4},{\i*sin(60)}) {};}
    \fi}
    \draw[fill=black] (0,{(2+1/2)*sin(60)}) circle (9pt);
    \draw[fill=black] (.75,{(2)*sin(60)}) circle (9pt);
    \draw[fill=black] (.75,{(3)*sin(60)}) circle (9pt);

    \draw[fill=lightgray] (1.5,{(2+1/2)*sin(60)}) circle (9pt);

\end{tikzpicture}
\end{figure}}

If Black's third marble is in position $d$ or $g$, and one of the two adjacent spaces ($a$, $c$, $h$, or $j$) is open, then Black can create the triangle in two moves. First they move into the empty $a$, $c$, $h$, or $j$ space, then they do a direct or inline move to create the triangle.  If both of the adjacent spaces are blocked by Gray, Black can create the triangle in three moves.  First they do an inline move to get their three marbles to the opposite end of the board, as shown in Figure~\ref{fig:crossingsecure}.  In response, Gray can block at most one of the corner spaces that are now adjacent to Black's third marble.  So Black can move into $a$, $c$, $h$, or $j$ with their second move, and do a direct or inline move to create the triangle with their third move.

{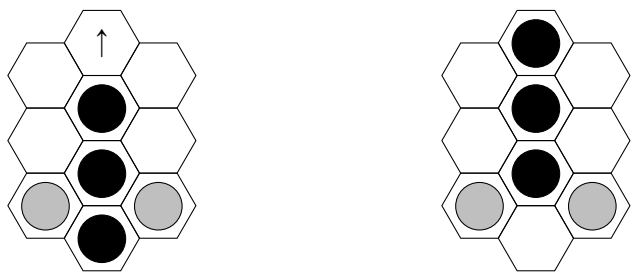
\begin{figure}[ht!]
\caption{Inline movement to create a triangle in three moves.}
\label{fig:crossingsecure}
\centering
\vspace{.5cm}
    \begin{tikzpicture} [hexa/.style= {shape=regular polygon,
                                   regular polygon sides=6,
                                   minimum size=1cm, draw,
                                   inner sep=0,anchor=south,}]
    \foreach \j in {0,...,2}{%
     \ifodd\j 
         \foreach \i in {0,...,3}{\node[hexa] at ({\j/2+\j/4},{(\i+1/2)*sin(60)}) {};}        
    \else
         \foreach \i in {1,...,3}{\node[hexa] at ({\j/2+\j/4},{\i*sin(60)}) {};}
    \fi}
    \draw[fill=black] (.75,{(1)*sin(60)}) circle (9pt);
    \draw[fill=black] (.75,{(2)*sin(60)}) circle (9pt);
    \draw[fill=black] (.75,{(3)*sin(60)}) circle (9pt);

    \draw[fill=lightgray] (1.5,{(1+1/2)*sin(60)}) circle (9pt);
    \draw[fill=lightgray] (0,{(1+1/2)*sin(60)}) circle (9pt);
    
    \node at (.75,{(4)*sin(60)}) {\large $\boldsymbol\uparrow$};
    
\end{tikzpicture}
\hspace{3cm}
     \begin{tikzpicture} [hexa/.style= {shape=regular polygon,
                                   regular polygon sides=6,
                                   minimum size=1cm, draw,
                                   inner sep=0,anchor=south,}]
    \foreach \j in {0,...,2}{%
     \ifodd\j 
         \foreach \i in {0,...,3}{\node[hexa] at ({\j/2+\j/4},{(\i+1/2)*sin(60)}) {};}        
    \else
         \foreach \i in {1,...,3}{\node[hexa] at ({\j/2+\j/4},{\i*sin(60)}) {};}
    \fi}
    \draw[fill=black] (.75,{(4)*sin(60)}) circle (9pt);
    \draw[fill=black] (.75,{(2)*sin(60)}) circle (9pt);
    \draw[fill=black] (.75,{(3)*sin(60)}) circle (9pt);

    \draw[fill=lightgray] (1.5,{(1+1/2)*sin(60)}) circle (9pt);
    \draw[fill=lightgray] (0,{(1+1/2)*sin(60)}) circle (9pt);

\end{tikzpicture}
\end{figure}}

Now let's consider gameplay that begins on a board with the Black triangle configuration. If Gray moves next, they must attempt to move their pieces into the three safe positions.  Failure to do so will allow Black to sumito and win on their next move.  If Black moves next, they will either sumito and win, or Gray's pieces are in the three safe positions.  In the latter case, Black can maintain control of the center by performing the inline move pictured in Figure~\ref{fig:2x2x3graysafe}. From this constellation, Gray's only safe response is to move their marble that's in space $b$ or $i$.  This allows Black to move their third marble into the space Gray just vacated, which creates an instance of the fork pictured way back in Figure~\ref{fig:forkfun}.  Since two nonadjacent Gray marbles are put in danger, Gray cannot move both of them to safety, and Black will sumito and win on their next move.

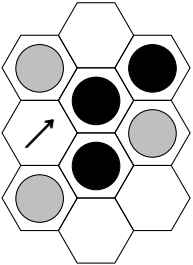
\begin{figure}[ht!]
\caption{Black move to secure a win when they must break their triangle.}
\label{fig:2x2x3graysafe}
\centering
\vspace{.5cm}
    \begin{tikzpicture} [hexa/.style= {shape=regular polygon,
                                   regular polygon sides=6,
                                   minimum size=1cm, draw,
                                   inner sep=0,anchor=south,}]
    \foreach \j in {0,...,2}{%
     \ifodd\j 
         \foreach \i in {0,...,3}{\node[hexa] at ({\j/2+\j/4},{(\i+1/2)*sin(60)}) {};}        
    \else
         \foreach \i in {1,...,3}{\node[hexa] at ({\j/2+\j/4},{\i*sin(60)}) {};}
    \fi}
    \draw[fill=black] (1.5,{(3+1/2)*sin(60)}) circle (9pt);
    
    \draw[fill=black] (.75,{(2)*sin(60)}) circle (9pt);
    \draw[fill=black] (.75,{(3)*sin(60)}) circle (9pt);
    
    \draw[fill=lightgray] (1.5,{(2+1/2)*sin(60)}) circle (9pt);
    \draw[fill=lightgray] (0,{(3+1/2)*sin(60)}) circle (9pt);
    \draw[fill=lightgray] (0,{(1+1/2)*sin(60)}) circle (9pt);
    
    \node at (0,{(2+1/2)*sin(60)}) {\large $\boldsymbol\nearrow$};
    
\end{tikzpicture}
\end{figure}

We have shown that if Black is able to get to a triangle configuration from $C$, then Black will win.  However, there is one situation where Black may lose despite having control of the middle of the board.  Suppose $C$ contains a subconfiguration isomorphic to the one pictured in Figure~\ref{fig:2x2x3neutral}, where the Gray marbles in positions $a$ and $b$ threaten the Black marble in position $c$.  If Black plays next on this board, they can force a win by either pushing off the third Gray marble or making an inline move to create the triangle configuration.  However, if Gray plays next, they can sumito the Black marble in position $c$ and win.  In other words, this board can be won by whomever plays next on it.  Thus, every constellation $C$ in which Black controls the center of the board is in outcome class $\LL$, except for those with the subconfiguration in Figure~\ref{fig:2x2x3neutral}, which are in outcome class $\N$.
\end{proof}

By Lemma~\ref{lem:negs}, we obtain the following corollary to Lemma~\ref{lem:223middle}.
\begin{corollary}
    Let $C$ be a $2\times2\times3$ Abalone constellation in which Gray occupies both of the middle two spaces. Then either $o(C)=\R$, or $-C$ contains the subconfiguration shown in Figure~\ref{fig:2x2x3neutral}. In the latter case, $o(C)=\N$.
    \label{cor:223middle}
\end{corollary}

Next, we will use Lemma~\ref{lem:223middle} and Corollary~\ref{cor:223middle} to determine the first two moves of an optimal game of $2\times 2\times 3$ Abalone.

    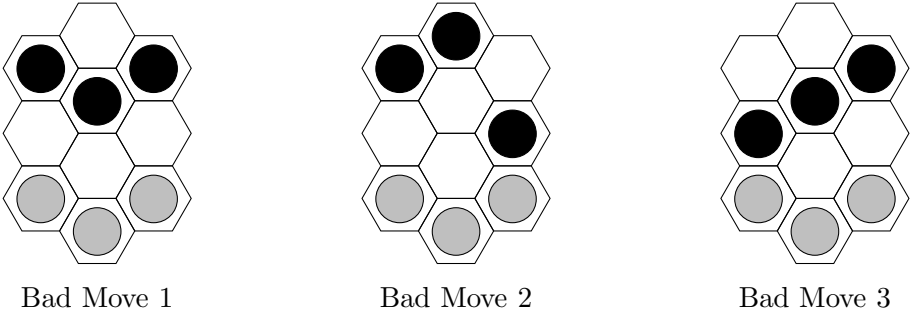
\begin{figure}[ht!]
\caption{Non-optimal first moves in $2\times 2\times 3$ Abalone.}
\label{fig:2x2x3badfirstmoves}
\centering
\vspace{.5cm}
    \begin{tikzpicture} [hexa/.style= {shape=regular polygon,
                                   regular polygon sides=6,
                                   minimum size=1cm, draw,
                                   inner sep=0,anchor=south,}]
    \foreach \j in {0,...,2}{%
     \ifodd\j 
         \foreach \i in {0,...,3}{\node[hexa] at ({\j/2+\j/4},{(\i+1/2)*sin(60)}) {};}        
    \else
         \foreach \i in {1,...,3}{\node[hexa] at ({\j/2+\j/4},{\i*sin(60)}) {};}
    \fi}
    \draw[fill=black] (0,{(3+1/2)*sin(60)}) circle (9pt);
    \draw[fill=black] (.75,{(3)*sin(60)}) circle (9pt);
    \draw[fill=black] (1.5,{(3+1/2)*sin(60)}) circle (9pt);
    \draw[fill=lightgray] (0,{(1+1/2)*sin(60)}) circle (9pt);
    \draw[fill=lightgray] (.75,{(1)*sin(60)}) circle (9pt);
    \draw[fill=lightgray] (1.5,{(1+1/2)*sin(60)}) circle (9pt);
    \node at (.75,0) {Bad Move 1};
\end{tikzpicture}
\hspace{2cm}
    \begin{tikzpicture} [hexa/.style= {shape=regular polygon,
                                   regular polygon sides=6,
                                   minimum size=1cm, draw,
                                   inner sep=0,anchor=south,}]
    \foreach \j in {0,...,2}{%
     \ifodd\j 
         \foreach \i in {0,...,3}{\node[hexa] at ({\j/2+\j/4},{(\i+1/2)*sin(60)}) {};}        
    \else
         \foreach \i in {1,...,3}{\node[hexa] at ({\j/2+\j/4},{\i*sin(60)}) {};}
    \fi}
    \draw[fill=black] (0,{(3+1/2)*sin(60)}) circle (9pt);
    \draw[fill=black] (.75,{(4)*sin(60)}) circle (9pt);
    \draw[fill=black] (1.5,{(2+1/2)*sin(60)}) circle (9pt);
    \draw[fill=lightgray] (0,{(1+1/2)*sin(60)}) circle (9pt);
    \draw[fill=lightgray] (.75,{(1)*sin(60)}) circle (9pt);
    \draw[fill=lightgray] (1.5,{(1+1/2)*sin(60)}) circle (9pt);
    \node at (.75,0) {Bad Move 2};
\end{tikzpicture}
\hspace{2cm}
    \begin{tikzpicture} [hexa/.style= {shape=regular polygon,
                                   regular polygon sides=6,
                                   minimum size=1cm, draw,
                                   inner sep=0,anchor=south,}]
    \foreach \j in {0,...,2}{%
     \ifodd\j 
         \foreach \i in {0,...,3}{\node[hexa] at ({\j/2+\j/4},{(\i+1/2)*sin(60)}) {};}        
    \else
         \foreach \i in {1,...,3}{\node[hexa] at ({\j/2+\j/4},{\i*sin(60)}) {};}
    \fi}
    \draw[fill=black] (0,{(2+1/2)*sin(60)}) circle (9pt);
    \draw[fill=black] (.75,{(3)*sin(60)}) circle (9pt);
    \draw[fill=black] (1.5,{(3+1/2)*sin(60)}) circle (9pt);
    \draw[fill=lightgray] (0,{(1+1/2)*sin(60)}) circle (9pt);
    \draw[fill=lightgray] (.75,{(1)*sin(60)}) circle (9pt);
    \draw[fill=lightgray] (1.5,{(1+1/2)*sin(60)}) circle (9pt);
    \node at (.75,0) {Bad Move 3};
\end{tikzpicture}
\end{figure}

\begin{lemma}
    If Black plays first in $2\times 2\times 3$ Abalone, then they will move to board $C1$ shown in Figure~\ref{fig:2x2x3optimalplayboards}.
    \label{lem:firstmove}
\end{lemma}
\begin{proof}
    Recall that the initial constellation for $2\times2\times3$ Abalone is $C0$.  From this board, Black has four non-isomorphic options, namely $C1$ and the three boards shown in Figure~\ref{fig:2x2x3badfirstmoves}.  

    Suppose Black makes Bad Move 1.  Then Gray can force one of the two sequences of moves shown in Figure~\ref{fig:badmove1seq}. After the first possible sequence of moves, Gray is able to safely create a triangle, which means Gray can force a win by Corollary~\ref{cor:223middle}.  In the second sequence of moves, Gray is able to create a fork.  Since Gray can force a win in both cases, Black will never make Bad Move 1.

    \begin{figure}[ht!]
\caption{First sequence (row 1) and second sequence (row 2) following Bad Move 1.}
\label{fig:badmove1seq}
\centering
\vspace{.5cm}
    \begin{tikzpicture} [hexa/.style= {shape=regular polygon,
                                   regular polygon sides=6,
                                   minimum size=1cm, draw,
                                   inner sep=0,anchor=south,}]
    \foreach \j in {0,...,2}{%
     \ifodd\j 
         \foreach \i in {0,...,3}{\node[hexa] at ({\j/2+\j/4},{(\i+1/2)*sin(60)}) {};}        
    \else
         \foreach \i in {1,...,3}{\node[hexa] at ({\j/2+\j/4},{\i*sin(60)}) {};}
    \fi}
    \draw[fill=black] (0,{(3+1/2)*sin(60)}) circle (9pt);
    \draw[fill=black] (.75,{(3)*sin(60)}) circle (9pt);
    \draw[fill=black] (1.5,{(3+1/2)*sin(60)}) circle (9pt);
    \draw[fill=lightgray] (0,{(1+1/2)*sin(60)}) circle (9pt);
    \draw[fill=lightgray] (.75,{(2)*sin(60)}) circle (9pt);
    \draw[fill=lightgray] (1.5,{(1+1/2)*sin(60)}) circle (9pt);
    \node at (.75,0) {Gray response};
\end{tikzpicture}
\hspace{.25cm}
    \begin{tikzpicture} [hexa/.style= {shape=regular polygon,
                                   regular polygon sides=6,
                                   minimum size=1cm, draw,
                                   inner sep=0,anchor=south,}]
    \foreach \j in {0,...,2}{%
     \ifodd\j 
         \foreach \i in {0,...,3}{\node[hexa] at ({\j/2+\j/4},{(\i+1/2)*sin(60)}) {};}        
    \else
         \foreach \i in {1,...,3}{\node[hexa] at ({\j/2+\j/4},{\i*sin(60)}) {};}
    \fi}
    \draw[fill=black] (0,{(3+1/2)*sin(60)}) circle (9pt);
    \draw[fill=black] (.75,{(4)*sin(60)}) circle (9pt);
    \draw[fill=black] (1.5,{(3+1/2)*sin(60)}) circle (9pt);
    \draw[fill=lightgray] (0,{(1+1/2)*sin(60)}) circle (9pt);
    \draw[fill=lightgray] (.75,{(2)*sin(60)}) circle (9pt);
    \draw[fill=lightgray] (1.5,{(1+1/2)*sin(60)}) circle (9pt);
    \node at (.75,0) {Black response 1};
\end{tikzpicture}
\hspace{.25cm}
    \begin{tikzpicture} [hexa/.style= {shape=regular polygon,
                                   regular polygon sides=6,
                                   minimum size=1cm, draw,
                                   inner sep=0,anchor=south,}]
    \foreach \j in {0,...,2}{%
     \ifodd\j 
         \foreach \i in {0,...,3}{\node[hexa] at ({\j/2+\j/4},{(\i+1/2)*sin(60)}) {};}        
    \else
         \foreach \i in {1,...,3}{\node[hexa] at ({\j/2+\j/4},{\i*sin(60)}) {};}
    \fi}
    \draw[fill=black] (0,{(3+1/2)*sin(60)}) circle (9pt);
    \draw[fill=black] (.75,{(4)*sin(60)}) circle (9pt);
    \draw[fill=black] (1.5,{(3+1/2)*sin(60)}) circle (9pt);
    \draw[fill=lightgray] (0,{(1+1/2)*sin(60)}) circle (9pt);
    \draw[fill=lightgray] (.75,{(3)*sin(60)}) circle (9pt);
    \draw[fill=lightgray] (1.5,{(2+1/2)*sin(60)}) circle (9pt);
    \node at (.75,0) {Gray response};
\end{tikzpicture}
\hspace{.25cm}
    \begin{tikzpicture} [hexa/.style= {shape=regular polygon,
                                   regular polygon sides=6,
                                   minimum size=1cm, draw,
                                   inner sep=0,anchor=south,}]
    \foreach \j in {0,...,2}{%
     \ifodd\j 
         \foreach \i in {0,...,3}{\node[hexa] at ({\j/2+\j/4},{(\i+1/2)*sin(60)}) {};}        
    \else
         \foreach \i in {1,...,3}{\node[hexa] at ({\j/2+\j/4},{\i*sin(60)}) {};}
    \fi}
    \draw[fill=black] (0,{(2+1/2)*sin(60)}) circle (9pt);
    \draw[fill=black] (.75,{(4)*sin(60)}) circle (9pt);
    \draw[fill=black] (1.5,{(3.5)*sin(60)}) circle (9pt);
    \draw[fill=lightgray] (0,{(1+1/2)*sin(60)}) circle (9pt);
    \draw[fill=lightgray] (.75,{(3)*sin(60)}) circle (9pt);
    \draw[fill=lightgray] (1.5,{(2.5)*sin(60)}) circle (9pt);
    \node at (.75,0) {Black response};
\end{tikzpicture}
\hspace{.25cm}
    \begin{tikzpicture} [hexa/.style= {shape=regular polygon,
                                   regular polygon sides=6,
                                   minimum size=1cm, draw,
                                   inner sep=0,anchor=south,}]
    \foreach \j in {0,...,2}{%
     \ifodd\j 
         \foreach \i in {0,...,3}{\node[hexa] at ({\j/2+\j/4},{(\i+1/2)*sin(60)}) {};}        
    \else
         \foreach \i in {1,...,3}{\node[hexa] at ({\j/2+\j/4},{\i*sin(60)}) {};}
    \fi}
    \draw[fill=black] (0,{(2+1/2)*sin(60)}) circle (9pt);
    \draw[fill=black] (.75,{(4)*sin(60)}) circle (9pt);
    \draw[fill=black] (1.5,{(3.5)*sin(60)}) circle (9pt);
    \draw[fill=lightgray] (.75,{(2)*sin(60)}) circle (9pt);
    \draw[fill=lightgray] (.75,{(3)*sin(60)}) circle (9pt);
    \draw[fill=lightgray] (1.5,{(2.5)*sin(60)}) circle (9pt);
    \node at (.75,0) {Gray win};
\end{tikzpicture}
\ \\
\vspace{.5cm}
    \begin{tikzpicture} [hexa/.style= {shape=regular polygon,
                                   regular polygon sides=6,
                                   minimum size=1cm, draw,
                                   inner sep=0,anchor=south,}]
    \foreach \j in {0,...,2}{%
     \ifodd\j 
         \foreach \i in {0,...,3}{\node[hexa] at ({\j/2+\j/4},{(\i+1/2)*sin(60)}) {};}        
    \else
         \foreach \i in {1,...,3}{\node[hexa] at ({\j/2+\j/4},{\i*sin(60)}) {};}
    \fi}
    \draw[fill=black] (0,{(3+1/2)*sin(60)}) circle (9pt);
    \draw[fill=black] (.75,{(3)*sin(60)}) circle (9pt);
    \draw[fill=black] (1.5,{(3+1/2)*sin(60)}) circle (9pt);
    \draw[fill=lightgray] (0,{(1+1/2)*sin(60)}) circle (9pt);
    \draw[fill=lightgray] (.75,{(2)*sin(60)}) circle (9pt);
    \draw[fill=lightgray] (1.5,{(1+1/2)*sin(60)}) circle (9pt);
    \node at (.75,0) {Gray response};
\end{tikzpicture}
\hspace{1cm}
    \begin{tikzpicture} [hexa/.style= {shape=regular polygon,
                                   regular polygon sides=6,
                                   minimum size=1cm, draw,
                                   inner sep=0,anchor=south,}]
    \foreach \j in {0,...,2}{%
     \ifodd\j 
         \foreach \i in {0,...,3}{\node[hexa] at ({\j/2+\j/4},{(\i+1/2)*sin(60)}) {};}        
    \else
         \foreach \i in {1,...,3}{\node[hexa] at ({\j/2+\j/4},{\i*sin(60)}) {};}
    \fi}
    \draw[fill=black] (0,{(3+1/2)*sin(60)}) circle (9pt);
    \draw[fill=black] (.75,{(4)*sin(60)}) circle (9pt);
    \draw[fill=black] (.75,{(3)*sin(60)}) circle (9pt);
    \draw[fill=lightgray] (0,{(1+1/2)*sin(60)}) circle (9pt);
    \draw[fill=lightgray] (.75,{(2)*sin(60)}) circle (9pt);
    \draw[fill=lightgray] (1.5,{(1+1/2)*sin(60)}) circle (9pt);
    \node at (.75,0) {Black response 2};
\end{tikzpicture}
\hspace{1cm}
    \begin{tikzpicture} [hexa/.style= {shape=regular polygon,
                                   regular polygon sides=6,
                                   minimum size=1cm, draw,
                                   inner sep=0,anchor=south,}]
    \foreach \j in {0,...,2}{%
     \ifodd\j 
         \foreach \i in {0,...,3}{\node[hexa] at ({\j/2+\j/4},{(\i+1/2)*sin(60)}) {};}        
    \else
         \foreach \i in {1,...,3}{\node[hexa] at ({\j/2+\j/4},{\i*sin(60)}) {};}
    \fi}
    \draw[fill=black] (0,{(3+1/2)*sin(60)}) circle (9pt);
    \draw[fill=black] (.75,{(4)*sin(60)}) circle (9pt);
    \draw[fill=black] (.75,{(3)*sin(60)}) circle (9pt);
    \draw[fill=lightgray] (0,{(1+1/2)*sin(60)}) circle (9pt);
    \draw[fill=lightgray] (.75,{(2)*sin(60)}) circle (9pt);
    \draw[fill=lightgray] (0,{(2+1/2)*sin(60)}) circle (9pt);
    \node at (1.5,{(2+1/2)*sin(60)}) {{\color{red} \huge\textbf{!}}};
    \node at (0,{(3+1/2)*sin(60)}) {{\color{red} \huge\textbf{!}}};
    \node at (.75,0) {Gray win};
\end{tikzpicture}
\end{figure}
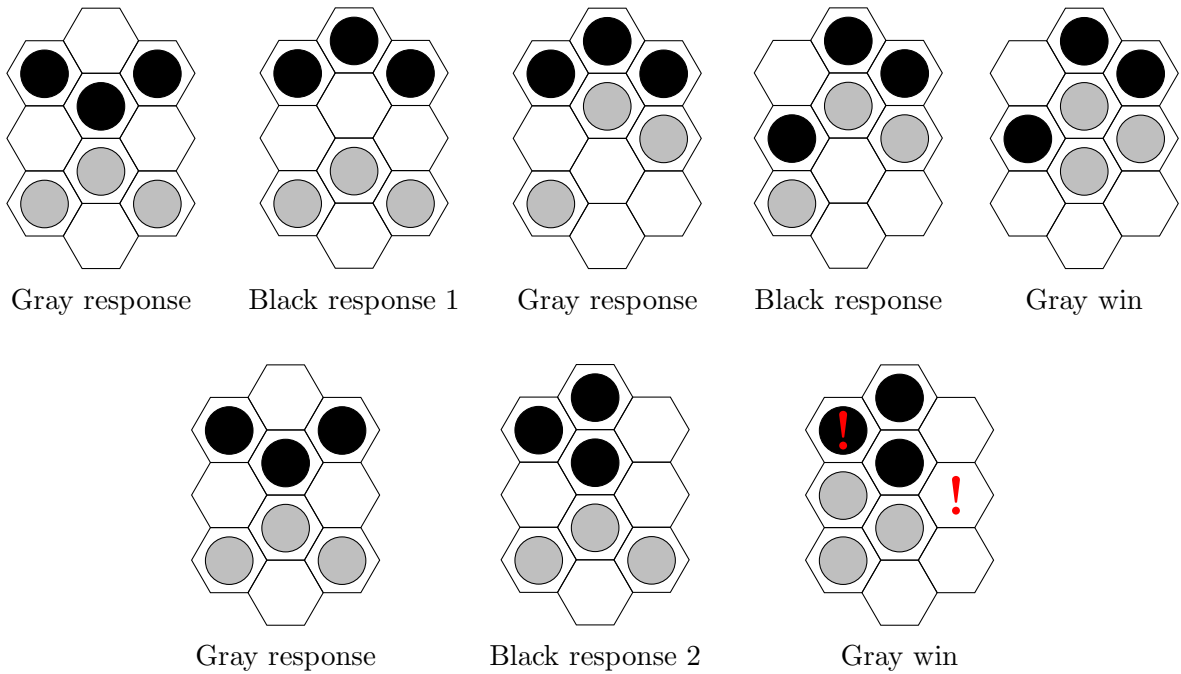

    Suppose Black makes Bad Move 2.  Then Gray can respond by choosing the option shown in Figure~\ref{fig:badmove2resp}. After this, Black must move their threatened piece out of danger.  However, doing so leads to either Black response 1  or Black response 2 from Figure~\ref{fig:badmove1seq}. We've already seen that Gray can force a win in both cases, so Black will never make Bad Move 2.

    \begin{figure}[ht!]
\caption{Gray response to Bad Move 2.}
\label{fig:badmove2resp}
\centering
\vspace{.5cm}
    \begin{tikzpicture} [hexa/.style= {shape=regular polygon,
                                   regular polygon sides=6,
                                   minimum size=1cm, draw,
                                   inner sep=0,anchor=south,}]
    \foreach \j in {0,...,2}{%
     \ifodd\j 
         \foreach \i in {0,...,3}{\node[hexa] at ({\j/2+\j/4},{(\i+1/2)*sin(60)}) {};}        
    \else
         \foreach \i in {1,...,3}{\node[hexa] at ({\j/2+\j/4},{\i*sin(60)}) {};}
    \fi}
    \draw[fill=black] (0,{(3+1/2)*sin(60)}) circle (9pt);
    \draw[fill=black] (.75,{(4)*sin(60)}) circle (9pt);
    \draw[fill=black] (1.5,{(2+1/2)*sin(60)}) circle (9pt);
    \draw[fill=lightgray] (0,{(1+1/2)*sin(60)}) circle (9pt);
    \draw[fill=lightgray] (.75,{(2)*sin(60)}) circle (9pt);
    \draw[fill=lightgray] (1.5,{(1+1/2)*sin(60)}) circle (9pt);
\end{tikzpicture}
\end{figure}
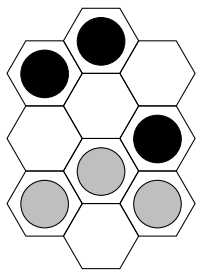

    Suppose Black makes Bad Move 3.  Then Gray can force one of the three sequences of moves shown in Figure~\ref{fig:badmove3seq1}. In each of these three sequences, Gray moves to a position where they can safely create a triangle.  Thus, by Corollary~\ref{cor:223middle}, Gray can force a win in all three cases, so Black will never make Bad Move 3.

    \begin{figure}[ht!]
\caption{First sequence (row 1), second sequence (row 2), and third sequence (row 3) following Bad Move 3.}
\label{fig:badmove3seq1}
\centering
\vspace{.5cm}
    \begin{tikzpicture} [hexa/.style= {shape=regular polygon,
                                   regular polygon sides=6,
                                   minimum size=1cm, draw,
                                   inner sep=0,anchor=south,}]
    \foreach \j in {0,...,2}{%
     \ifodd\j 
         \foreach \i in {0,...,3}{\node[hexa] at ({\j/2+\j/4},{(\i+1/2)*sin(60)}) {};}        
    \else
         \foreach \i in {1,...,3}{\node[hexa] at ({\j/2+\j/4},{\i*sin(60)}) {};}
    \fi}
    \draw[fill=black] (0,{(2+1/2)*sin(60)}) circle (9pt);
    \draw[fill=black] (.75,{(3)*sin(60)}) circle (9pt);
    \draw[fill=black] (1.5,{(3+1/2)*sin(60)}) circle (9pt);
    \draw[fill=lightgray] (.75,{(2)*sin(60)}) circle (9pt);
    \draw[fill=lightgray] (.75,{(1)*sin(60)}) circle (9pt);
    \draw[fill=lightgray] (1.5,{(1+1/2)*sin(60)}) circle (9pt);
    \node at (.75,0) {Gray response};
\end{tikzpicture}
\hspace{1cm}
    \begin{tikzpicture} [hexa/.style= {shape=regular polygon,
                                   regular polygon sides=6,
                                   minimum size=1cm, draw,
                                   inner sep=0,anchor=south,}]
    \foreach \j in {0,...,2}{%
     \ifodd\j 
         \foreach \i in {0,...,3}{\node[hexa] at ({\j/2+\j/4},{(\i+1/2)*sin(60)}) {};}        
    \else
         \foreach \i in {1,...,3}{\node[hexa] at ({\j/2+\j/4},{\i*sin(60)}) {};}
    \fi}
    \draw[fill=black] (0,{(3+1/2)*sin(60)}) circle (9pt);
    \draw[fill=black] (.75,{(3)*sin(60)}) circle (9pt);
    \draw[fill=black] (1.5,{(3+1/2)*sin(60)}) circle (9pt);
    \draw[fill=lightgray] (.75,{(2)*sin(60)}) circle (9pt);
    \draw[fill=lightgray] (.75,{(1)*sin(60)}) circle (9pt);
    \draw[fill=lightgray] (1.5,{(1+1/2)*sin(60)}) circle (9pt);
    \node at (.75,0) {Black response 1};
\end{tikzpicture}
\hspace{1cm}
    \begin{tikzpicture} [hexa/.style= {shape=regular polygon,
                                   regular polygon sides=6,
                                   minimum size=1cm, draw,
                                   inner sep=0,anchor=south,}]
    \foreach \j in {0,...,2}{%
     \ifodd\j 
         \foreach \i in {0,...,3}{\node[hexa] at ({\j/2+\j/4},{(\i+1/2)*sin(60)}) {};}        
    \else
         \foreach \i in {1,...,3}{\node[hexa] at ({\j/2+\j/4},{\i*sin(60)}) {};}
    \fi}
    \draw[fill=black] (0,{(3+1/2)*sin(60)}) circle (9pt);
    \draw[fill=black] (.75,{(4)*sin(60)}) circle (9pt);
    \draw[fill=black] (1.5,{(3+1/2)*sin(60)}) circle (9pt);
    \draw[fill=lightgray] (.75,{(3)*sin(60)}) circle (9pt);
    \draw[fill=lightgray] (.75,{(2)*sin(60)}) circle (9pt);
    \draw[fill=lightgray] (1.5,{(1+1/2)*sin(60)}) circle (9pt);
    \node at (.75,0) {Gray response};
\end{tikzpicture}
\ \\
\vspace{.5cm}
    \begin{tikzpicture} [hexa/.style= {shape=regular polygon,
                                   regular polygon sides=6,
                                   minimum size=1cm, draw,
                                   inner sep=0,anchor=south,}]
    \foreach \j in {0,...,2}{%
     \ifodd\j 
         \foreach \i in {0,...,3}{\node[hexa] at ({\j/2+\j/4},{(\i+1/2)*sin(60)}) {};}        
    \else
         \foreach \i in {1,...,3}{\node[hexa] at ({\j/2+\j/4},{\i*sin(60)}) {};}
    \fi}
    \draw[fill=black] (0,{(2+1/2)*sin(60)}) circle (9pt);
    \draw[fill=black] (.75,{(3)*sin(60)}) circle (9pt);
    \draw[fill=black] (1.5,{(3+1/2)*sin(60)}) circle (9pt);
    \draw[fill=lightgray] (.75,{(2)*sin(60)}) circle (9pt);
    \draw[fill=lightgray] (.75,{(1)*sin(60)}) circle (9pt);
    \draw[fill=lightgray] (1.5,{(1+1/2)*sin(60)}) circle (9pt);
    \node at (.75,0) {Gray response};
\end{tikzpicture}
\hspace{1cm}
    \begin{tikzpicture} [hexa/.style= {shape=regular polygon,
                                   regular polygon sides=6,
                                   minimum size=1cm, draw,
                                   inner sep=0,anchor=south,}]
    \foreach \j in {0,...,2}{%
     \ifodd\j 
         \foreach \i in {0,...,3}{\node[hexa] at ({\j/2+\j/4},{(\i+1/2)*sin(60)}) {};}        
    \else
         \foreach \i in {1,...,3}{\node[hexa] at ({\j/2+\j/4},{\i*sin(60)}) {};}
    \fi}
    \draw[fill=black] (0,{(3+1/2)*sin(60)}) circle (9pt);
    \draw[fill=black] (.75,{(4)*sin(60)}) circle (9pt);
    \draw[fill=black] (1.5,{(3+1/2)*sin(60)}) circle (9pt);
    \draw[fill=lightgray] (.75,{(2)*sin(60)}) circle (9pt);
    \draw[fill=lightgray] (.75,{(1)*sin(60)}) circle (9pt);
    \draw[fill=lightgray] (1.5,{(1+1/2)*sin(60)}) circle (9pt);
    \node at (.75,0) {Black response 2};
\end{tikzpicture}
\hspace{1cm}
    \begin{tikzpicture} [hexa/.style= {shape=regular polygon,
                                   regular polygon sides=6,
                                   minimum size=1cm, draw,
                                   inner sep=0,anchor=south,}]
    \foreach \j in {0,...,2}{%
     \ifodd\j 
         \foreach \i in {0,...,3}{\node[hexa] at ({\j/2+\j/4},{(\i+1/2)*sin(60)}) {};}        
    \else
         \foreach \i in {1,...,3}{\node[hexa] at ({\j/2+\j/4},{\i*sin(60)}) {};}
    \fi}
    \draw[fill=black] (0,{(3+1/2)*sin(60)}) circle (9pt);
    \draw[fill=black] (.75,{(4)*sin(60)}) circle (9pt);
    \draw[fill=black] (1.5,{(3+1/2)*sin(60)}) circle (9pt);
    \draw[fill=lightgray] (.75,{(3)*sin(60)}) circle (9pt);
    \draw[fill=lightgray] (.75,{(2)*sin(60)}) circle (9pt);
    \draw[fill=lightgray] (1.5,{(1+1/2)*sin(60)}) circle (9pt);
    \node at (.75,0) {Gray response};
\end{tikzpicture}
\ \\
\vspace{.5cm}
    \begin{tikzpicture} [hexa/.style= {shape=regular polygon,
                                   regular polygon sides=6,
                                   minimum size=1cm, draw,
                                   inner sep=0,anchor=south,}]
    \foreach \j in {0,...,2}{%
     \ifodd\j 
         \foreach \i in {0,...,3}{\node[hexa] at ({\j/2+\j/4},{(\i+1/2)*sin(60)}) {};}        
    \else
         \foreach \i in {1,...,3}{\node[hexa] at ({\j/2+\j/4},{\i*sin(60)}) {};}
    \fi}
    \draw[fill=black] (0,{(2+1/2)*sin(60)}) circle (9pt);
    \draw[fill=black] (.75,{(3)*sin(60)}) circle (9pt);
    \draw[fill=black] (1.5,{(3+1/2)*sin(60)}) circle (9pt);
    \draw[fill=lightgray] (.75,{(2)*sin(60)}) circle (9pt);
    \draw[fill=lightgray] (.75,{(1)*sin(60)}) circle (9pt);
    \draw[fill=lightgray] (1.5,{(1+1/2)*sin(60)}) circle (9pt);
    \node at (.75,0) {Gray response};
\end{tikzpicture}
\hspace{1cm}
    \begin{tikzpicture} [hexa/.style= {shape=regular polygon,
                                   regular polygon sides=6,
                                   minimum size=1cm, draw,
                                   inner sep=0,anchor=south,}]
    \foreach \j in {0,...,2}{%
     \ifodd\j 
         \foreach \i in {0,...,3}{\node[hexa] at ({\j/2+\j/4},{(\i+1/2)*sin(60)}) {};}        
    \else
         \foreach \i in {1,...,3}{\node[hexa] at ({\j/2+\j/4},{\i*sin(60)}) {};}
    \fi}
    \draw[fill=black] (0,{(1+1/2)*sin(60)}) circle (9pt);
    \draw[fill=black] (.75,{(3)*sin(60)}) circle (9pt);
    \draw[fill=black] (1.5,{(3+1/2)*sin(60)}) circle (9pt);
    \draw[fill=lightgray] (.75,{(2)*sin(60)}) circle (9pt);
    \draw[fill=lightgray] (.75,{(1)*sin(60)}) circle (9pt);
    \draw[fill=lightgray] (1.5,{(1+1/2)*sin(60)}) circle (9pt);
    \node at (.75,0) {Black response 3};
\end{tikzpicture}
\hspace{1cm}
    \begin{tikzpicture} [hexa/.style= {shape=regular polygon,
                                   regular polygon sides=6,
                                   minimum size=1cm, draw,
                                   inner sep=0,anchor=south,}]
    \foreach \j in {0,...,2}{%
     \ifodd\j 
         \foreach \i in {0,...,3}{\node[hexa] at ({\j/2+\j/4},{(\i+1/2)*sin(60)}) {};}        
    \else
         \foreach \i in {1,...,3}{\node[hexa] at ({\j/2+\j/4},{\i*sin(60)}) {};}
    \fi}
    \draw[fill=black] (0,{(1+1/2)*sin(60)}) circle (9pt);
    \draw[fill=black] (.75,{(4)*sin(60)}) circle (9pt);
    \draw[fill=black] (1.5,{(3+1/2)*sin(60)}) circle (9pt);
    \draw[fill=lightgray] (.75,{(2)*sin(60)}) circle (9pt);
    \draw[fill=lightgray] (.75,{(3)*sin(60)}) circle (9pt);
    \draw[fill=lightgray] (1.5,{(1+1/2)*sin(60)}) circle (9pt);
    \node at (.75,0) {Gray response};
\end{tikzpicture}
\end{figure}
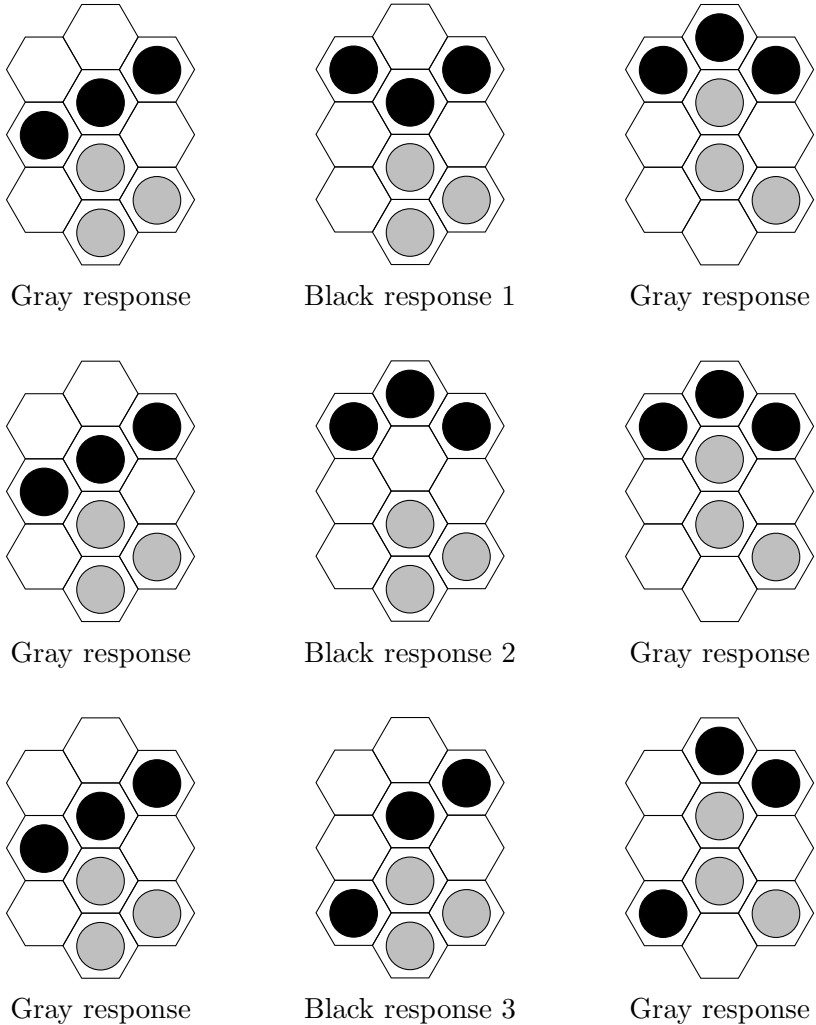

    Since we have determined that Gray can force a win when Black makes Bad Moves 1-3, the only logical first move for Black to make is $C1$.
\end{proof}

\begin{lemma}
    If Gray plays next on board $C1$ shown in Figure~\ref{fig:2x2x3optimalplayboards}, then they will move to either board $C2$ or $C3$.
    \label{lem:optimalgrayresponses}
\end{lemma}
\begin{proof}
    From board $C1$, Gray has seven non-isomorphic options.  Five of these options are $C2$, $C3$, and the the three boards in Figure~\ref{fig:2x2x3badgrayresponses}.  The two options not depicted would put one of Gray's pieces in danger, allowing Black to immediately push a Gray piece off the board.  These are not logical moves for Gray to make, so we will not consider them.

    \begin{figure}[ht!]
\caption{Non-optimal responses in $2\times 2\times 3$ Abalone.}
\label{fig:2x2x3badgrayresponses}
\centering
\vspace{.5cm}
    \begin{tikzpicture} [hexa/.style= {shape=regular polygon,
                                   regular polygon sides=6,
                                   minimum size=1cm, draw,
                                   inner sep=0,anchor=south,}]
    \foreach \j in {0,...,2}{%
     \ifodd\j 
         \foreach \i in {0,...,3}{\node[hexa] at ({\j/2+\j/4},{(\i+1/2)*sin(60)}) {};}        
    \else
         \foreach \i in {1,...,3}{\node[hexa] at ({\j/2+\j/4},{\i*sin(60)}) {};}
    \fi}
    \draw[fill=black] (0,{(3+1/2)*sin(60)}) circle (9pt);
    \draw[fill=black] (.75,{(4)*sin(60)}) circle (9pt);
    \draw[fill=black] (.75,{(3)*sin(60)}) circle (9pt);
    \draw[fill=lightgray] (0,{(2+1/2)*sin(60)}) circle (9pt);
    \draw[fill=lightgray] (.75,{(2)*sin(60)}) circle (9pt);
    \draw[fill=lightgray] (1.5,{(1+1/2)*sin(60)}) circle (9pt);
    \node at (.75,0) {Bad Response 1};
\end{tikzpicture}
\hspace{2cm}
    \begin{tikzpicture} [hexa/.style= {shape=regular polygon,
                                   regular polygon sides=6,
                                   minimum size=1cm, draw,
                                   inner sep=0,anchor=south,}]
    \foreach \j in {0,...,2}{%
     \ifodd\j 
         \foreach \i in {0,...,3}{\node[hexa] at ({\j/2+\j/4},{(\i+1/2)*sin(60)}) {};}        
    \else
         \foreach \i in {1,...,3}{\node[hexa] at ({\j/2+\j/4},{\i*sin(60)}) {};}
    \fi}
    \draw[fill=black] (0,{(3+1/2)*sin(60)}) circle (9pt);
    \draw[fill=black] (.75,{(4)*sin(60)}) circle (9pt);
    \draw[fill=black] (.75,{(3)*sin(60)}) circle (9pt);
    \draw[fill=lightgray] (0,{(2+1/2)*sin(60)}) circle (9pt);
    \draw[fill=lightgray] (.75,{(1)*sin(60)}) circle (9pt);
    \draw[fill=lightgray] (1.5,{(1+1/2)*sin(60)}) circle (9pt);
    \node at (.75,0) {Bad Response 2};
\end{tikzpicture}
\hspace{2cm}
    \begin{tikzpicture} [hexa/.style= {shape=regular polygon,
                                   regular polygon sides=6,
                                   minimum size=1cm, draw,
                                   inner sep=0,anchor=south,}]
    \foreach \j in {0,...,2}{%
     \ifodd\j 
         \foreach \i in {0,...,3}{\node[hexa] at ({\j/2+\j/4},{(\i+1/2)*sin(60)}) {};}        
    \else
         \foreach \i in {1,...,3}{\node[hexa] at ({\j/2+\j/4},{\i*sin(60)}) {};}
    \fi}
    \draw[fill=black] (0,{(3+1/2)*sin(60)}) circle (9pt);
    \draw[fill=black] (.75,{(3)*sin(60)}) circle (9pt);
    \draw[fill=black] (.75,{(4)*sin(60)}) circle (9pt);
    \draw[fill=lightgray] (0,{(1+1/2)*sin(60)}) circle (9pt);
    \draw[fill=lightgray] (.75,{(2)*sin(60)}) circle (9pt);
    \draw[fill=lightgray] (1.5,{(1+1/2)*sin(60)}) circle (9pt);
    \node at (.75,0) {Bad Response 3};
\end{tikzpicture}
\end{figure}
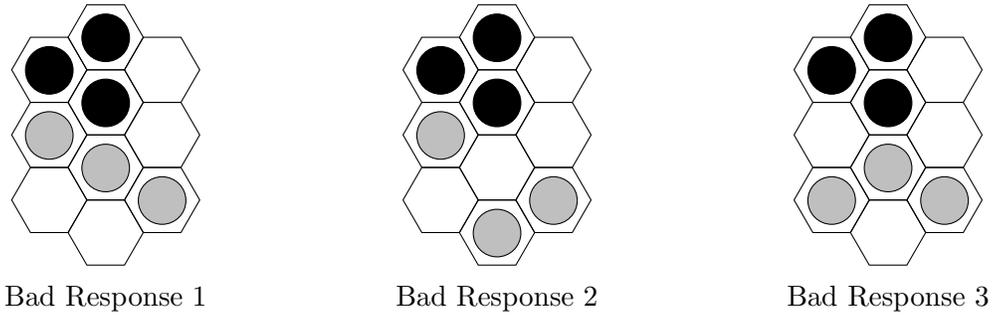

    Suppose Gray makes Bad Response 1 or Bad Response 2. Then Black can take the option shown in Figure~\ref{fig:badresp1seq}.  Gray must respond by moving its endangered piece, which allows Black to safely create a triangle.  Thus, by Lemma~\ref{lem:223middle}, Black can force a win, so Gray will never make Bad Response 1 or 2.

    \begin{figure}[ht!]
\caption{Sequence of moves following Bad Response 1 or 2.}
\label{fig:badresp1seq}
\centering
\vspace{.5cm}
    \begin{tikzpicture} [hexa/.style= {shape=regular polygon,
                                   regular polygon sides=6,
                                   minimum size=1cm, draw,
                                   inner sep=0,anchor=south,}]
    \foreach \j in {0,...,2}{%
     \ifodd\j 
         \foreach \i in {0,...,3}{\node[hexa] at ({\j/2+\j/4},{(\i+1/2)*sin(60)}) {};}        
    \else
         \foreach \i in {1,...,3}{\node[hexa] at ({\j/2+\j/4},{\i*sin(60)}) {};}
    \fi}
    \draw[fill=black] (0,{(3+1/2)*sin(60)}) circle (9pt);
    \draw[fill=black] (.75,{(3)*sin(60)}) circle (9pt);
    \draw[fill=black] (.75,{(2)*sin(60)}) circle (9pt);
    \draw[fill=lightgray] (0,{(2+1/2)*sin(60)}) circle (9pt);
    \draw[fill=lightgray] (.75,{(1)*sin(60)}) circle (9pt);
    \draw[fill=lightgray] (1.5,{(1+1/2)*sin(60)}) circle (9pt);
    \node at (.75,0) {Black response};
\end{tikzpicture}
\hspace{1cm}
    \begin{tikzpicture} [hexa/.style= {shape=regular polygon,
                                   regular polygon sides=6,
                                   minimum size=1cm, draw,
                                   inner sep=0,anchor=south,}]
    \foreach \j in {0,...,2}{%
     \ifodd\j 
         \foreach \i in {0,...,3}{\node[hexa] at ({\j/2+\j/4},{(\i+1/2)*sin(60)}) {};}        
    \else
         \foreach \i in {1,...,3}{\node[hexa] at ({\j/2+\j/4},{\i*sin(60)}) {};}
    \fi}
    \draw[fill=black] (0,{(3+1/2)*sin(60)}) circle (9pt);
    \draw[fill=black] (.75,{(3)*sin(60)}) circle (9pt);
    \draw[fill=black] (.75,{(2)*sin(60)}) circle (9pt);
    \draw[fill=lightgray] (0,{(2+1/2)*sin(60)}) circle (9pt);
    \draw[fill=lightgray] (0,{(1.5)*sin(60)}) circle (9pt);
    \draw[fill=lightgray] (1.5,{(1+1/2)*sin(60)}) circle (9pt);
    \node at (.75,0) {Gray response };
\end{tikzpicture}
\hspace{1cm}
    \begin{tikzpicture} [hexa/.style= {shape=regular polygon,
                                   regular polygon sides=6,
                                   minimum size=1cm, draw,
                                   inner sep=0,anchor=south,}]
    \foreach \j in {0,...,2}{%
     \ifodd\j 
         \foreach \i in {0,...,3}{\node[hexa] at ({\j/2+\j/4},{(\i+1/2)*sin(60)}) {};}        
    \else
         \foreach \i in {1,...,3}{\node[hexa] at ({\j/2+\j/4},{\i*sin(60)}) {};}
    \fi}
    \draw[fill=black] (.75,{(3)*sin(60)}) circle (9pt);
    \draw[fill=black] (1.5,{(2.5)*sin(60)}) circle (9pt);
    \draw[fill=black] (.75,{(2)*sin(60)}) circle (9pt);
    \draw[fill=lightgray] (0,{(2+1/2)*sin(60)}) circle (9pt);
    \draw[fill=lightgray] (0,{(1.5)*sin(60)}) circle (9pt);
    \draw[fill=lightgray] (1.5,{(1+1/2)*sin(60)}) circle (9pt);
    \node at (.75,0) {Black triangle};
\end{tikzpicture}
\end{figure}
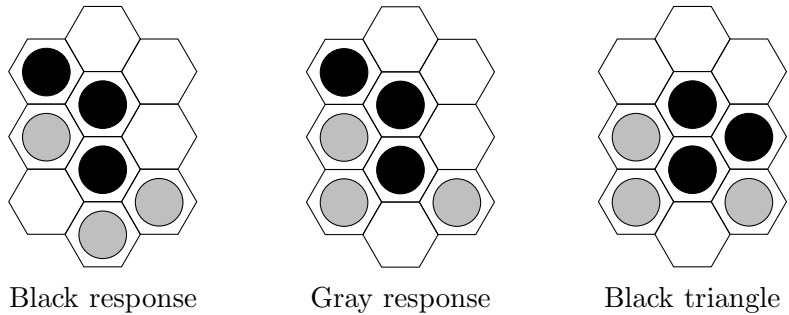

    Suppose Gray makes Bad Response 3. Then Black can take the option shown in Figure~\ref{fig:badresp3seq}.  This constellation is self negative, so Black and Gray have the same options on this board.  We saw in Figure~\ref{fig:badmove1seq} that Gray can force a win if they choose to go to this board.  By a similar argument, Black can force a win by choosing this option, meaning Gray should never make Bad Response 3.

    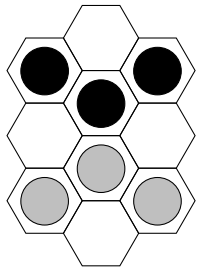
\begin{figure}[ht!]
\caption{Black option following Bad Response 3.}
\label{fig:badresp3seq}
\centering
\vspace{.5cm}
    \begin{tikzpicture} [hexa/.style= {shape=regular polygon,
                                   regular polygon sides=6,
                                   minimum size=1cm, draw,
                                   inner sep=0,anchor=south,}]
    \foreach \j in {0,...,2}{%
     \ifodd\j 
         \foreach \i in {0,...,3}{\node[hexa] at ({\j/2+\j/4},{(\i+1/2)*sin(60)}) {};}        
    \else
         \foreach \i in {1,...,3}{\node[hexa] at ({\j/2+\j/4},{\i*sin(60)}) {};}
    \fi}
    \draw[fill=black] (0,{(3+1/2)*sin(60)}) circle (9pt);
    \draw[fill=black] (.75,{(3)*sin(60)}) circle (9pt);
    \draw[fill=black] (1.5,{(3.5)*sin(60)}) circle (9pt);
    \draw[fill=lightgray] (0,{(1+1/2)*sin(60)}) circle (9pt);
    \draw[fill=lightgray] (.75,{(2)*sin(60)}) circle (9pt);
    \draw[fill=lightgray] (1.5,{(1+1/2)*sin(60)}) circle (9pt);
\end{tikzpicture}
\end{figure}

    Since we have determined that Black can force a win when Gray makes Bad Responses 1-3, the only logical moves for Gray to make are $C2$ and $C3$.
\end{proof}

We are now ready to determine the outcome class of the starting configuration of $2\times2\times3$ Abalone and therefore weakly solve the game.  

\begin{theorem}
$2\times 2\times 3$ Abalone is weakly solved.  If $C0$ denotes the initial constellation in a $2\times2\times3$ Abalone game, $o(C0) = \D$.
    \label{thm:223weaksolve}
\end{theorem}
\begin{proof}
Suppose Black and Gray are playing a game of $2\times 2\times 3$ Abalone and Black goes first.  By convention, the game begins with constellation $C0$ as depicted in Figure~\ref{fig:2x2x3optimalplayboards}.  By Lemma~\ref{lem:firstmove}, we know that Black will choose to go to board $C1$.  And by Lemma~\ref{lem:optimalgrayresponses}, we know that Gray will respond by going to either $C2$ or $C3$.  

From this point on, the players need to choose options which will \emph{not} lead to a win for their opponent.  Clearly, this means they will never choose an option where their opponent can immediately push to win.  It also means they must avoid moves which allow their opponent to create a fork.  Thus, by Lemma~\ref{lem:223middle} and Corollary~\ref{cor:223middle}, a player will never make a move that allows their opponent to safely create a triangle subconfiguration.  This means that a player will not move any of their pieces out of the middle column unless doing so creates a threat similar to the one shown in the neutral subconfiguration in Figure~\ref{fig:2x2x3neutral}. Creating this sort of a threat forces the opponent to retreat to safety, which allows the first player to reclaim their positions in the center column and prevent their opponent from controlling the middle two board positions.  However, this sort of move is not always a safe choice.  For example, Figure~\ref{fig:forkcreationsequence} shows an instance where leaving the middle column while threatening one of the opponent's pieces leads to a loss.

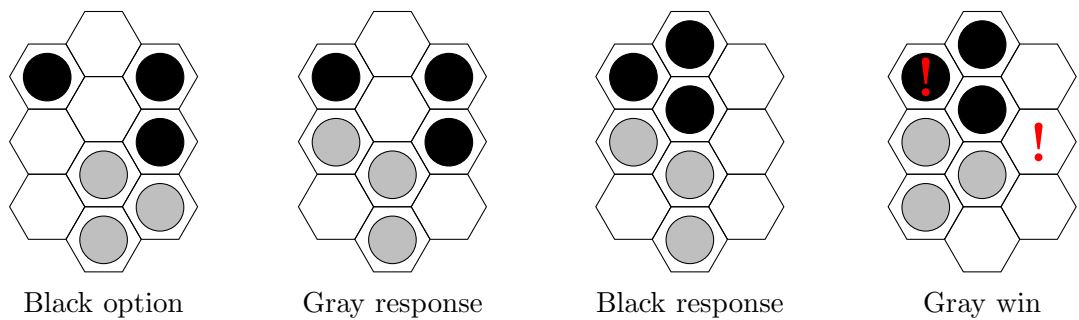
\begin{figure}[ht!]
\caption{A sequence of moves which leads to a win for Gray.}
\label{fig:forkcreationsequence}
\centering
\vspace{.5cm}
    \begin{tikzpicture} [hexa/.style= {shape=regular polygon,
                                   regular polygon sides=6,
                                   minimum size=1cm, draw,
                                   inner sep=0,anchor=south,}]
    \foreach \j in {0,...,2}{%
     \ifodd\j 
         \foreach \i in {0,...,3}{\node[hexa] at ({\j/2+\j/4},{(\i+1/2)*sin(60)}) {};}        
    \else
         \foreach \i in {1,...,3}{\node[hexa] at ({\j/2+\j/4},{\i*sin(60)}) {};}
    \fi}
    \draw[fill=black] (0,{(3+1/2)*sin(60)}) circle (9pt);
    \draw[fill=black] (1.5,{(3.5)*sin(60)}) circle (9pt);
    \draw[fill=black] (1.5,{(2.5)*sin(60)}) circle (9pt);
    \draw[fill=lightgray] (.75,{(1)*sin(60)}) circle (9pt);
    \draw[fill=lightgray] (.75,{(2)*sin(60)}) circle (9pt);
    \draw[fill=lightgray] (1.5,{(1+1/2)*sin(60)}) circle (9pt);
    \node at (.75,0) {Black option};
\end{tikzpicture}
\hspace{1cm}
    \begin{tikzpicture} [hexa/.style= {shape=regular polygon,
                                   regular polygon sides=6,
                                   minimum size=1cm, draw,
                                   inner sep=0,anchor=south,}]
    \foreach \j in {0,...,2}{%
     \ifodd\j 
         \foreach \i in {0,...,3}{\node[hexa] at ({\j/2+\j/4},{(\i+1/2)*sin(60)}) {};}        
    \else
         \foreach \i in {1,...,3}{\node[hexa] at ({\j/2+\j/4},{\i*sin(60)}) {};}
    \fi}
    \draw[fill=black] (0,{(3+1/2)*sin(60)}) circle (9pt);
    \draw[fill=black] (1.5,{(3.5)*sin(60)}) circle (9pt);
    \draw[fill=black] (1.5,{(2.5)*sin(60)}) circle (9pt);
    \draw[fill=lightgray] (0,{(2+1/2)*sin(60)}) circle (9pt);
    \draw[fill=lightgray] (.75,{(1)*sin(60)}) circle (9pt);
    \draw[fill=lightgray] (.75,{(2)*sin(60)}) circle (9pt);
    \node at (.75,0) {Gray response};
\end{tikzpicture}
\hspace{1cm}
    \begin{tikzpicture} [hexa/.style= {shape=regular polygon,
                                   regular polygon sides=6,
                                   minimum size=1cm, draw,
                                   inner sep=0,anchor=south,}]
    \foreach \j in {0,...,2}{%
     \ifodd\j 
         \foreach \i in {0,...,3}{\node[hexa] at ({\j/2+\j/4},{(\i+1/2)*sin(60)}) {};}        
    \else
         \foreach \i in {1,...,3}{\node[hexa] at ({\j/2+\j/4},{\i*sin(60)}) {};}
    \fi}
    \draw[fill=black] (0,{(3+1/2)*sin(60)}) circle (9pt);
    \draw[fill=black] (.75,{(4)*sin(60)}) circle (9pt);
    \draw[fill=black] (.75,{(3)*sin(60)}) circle (9pt);
    \draw[fill=lightgray] (0,{(2+1/2)*sin(60)}) circle (9pt);
    \draw[fill=lightgray] (.75,{(1)*sin(60)}) circle (9pt);
    \draw[fill=lightgray] (.75,{(2)*sin(60)}) circle (9pt);
    \node at (.75,0) {Black response};
\end{tikzpicture}
\hspace{1cm}
    \begin{tikzpicture} [hexa/.style= {shape=regular polygon,
                                   regular polygon sides=6,
                                   minimum size=1cm, draw,
                                   inner sep=0,anchor=south,}]
    \foreach \j in {0,...,2}{%
     \ifodd\j 
         \foreach \i in {0,...,3}{\node[hexa] at ({\j/2+\j/4},{(\i+1/2)*sin(60)}) {};}        
    \else
         \foreach \i in {1,...,3}{\node[hexa] at ({\j/2+\j/4},{\i*sin(60)}) {};}
    \fi}
    \draw[fill=black] (0,{(3+1/2)*sin(60)}) circle (9pt);
    \draw[fill=black] (.75,{(4)*sin(60)}) circle (9pt);
    \draw[fill=black] (.75,{(3)*sin(60)}) circle (9pt);
    \draw[fill=lightgray] (0,{(2+1/2)*sin(60)}) circle (9pt);
    \draw[fill=lightgray] (.75,{(2)*sin(60)}) circle (9pt);
    \draw[fill=lightgray] (0,{(1.5)*sin(60)}) circle (9pt);
    \node at (1.5,{(2+1/2)*sin(60)}) {{\color{red} \huge\textbf{!}}};
    \node at (0,{(3+1/2)*sin(60)}) {{\color{red} \huge\textbf{!}}};
    \node at (.75,0) {Gray win};
\end{tikzpicture}
\end{figure}

Notice that Black moves out of the middle column to threaten a Gray piece.  Gray must respond by moving their endangered piece to safety.  Then Black must respond by moving back into the middle column to prevent Gray from creating a triangle subconfiguration and winning.  However, Gray is then able to create a fork which threatens one of Black's pieces and the space Black must move into to protect the threatened piece.  Therefore, Black will never choose the option shown in Figure~\ref{fig:forkcreationsequence}, and by a similar argument, Gray will never choose its negative.  

As a result of these gameplay conventions, an optimal game of $2\times2\times3$ Abalone will involve only the following ``safe'' options:
\begin{center}
\begin{tabular}{rcll}
    $C0^L$ & $ =$ & $\{C1\}$& \\
    \hline
    $C1^R$ & $ =$ & $\{C2, C3\}$&\\
    \hline
    $C2^L$ & $ =$ & $\{C4, C5\}$&\\
    $C3^L$ & $ =$ & $ \{C4\}$& \\
    \hline
    $C4^R$ & $ =$ & $ \{C6\}$& \\
    $C5^R$ & $ =$ & $ \{C7\}$& \\
    \hline
    $C6^L$ & $ =$ & $ \{C8\}$& \\
    $C7^L$ & $ =$ & $ \{C6, C8\}$& \\
    \hline
    $C6^R$ & $ =$ & $ \{-C8 \}$& \\
    $C8^R$ & $ =$ & $ \{C2, C3\}$&  (Loop back for $C2^L$ and $C3^L$) \\
    \hline
    $-C8^L$ & $ =$ & $ \{C2, C3\}$& \\
    \hline
    $C2^R$ & $ =$ & $ \{-C4, -C5\}$&\\
    $C3^R$ & $ =$ & $ \{-C4\}$& \\
    \hline
    $-C4^L$ & $ =$ & $ \{C6\}$& (Loop back for $C6^R$)\\
    $-C5^L$ & $ =$ & $ \{-C7\}$& \\
    \hline
    $-C7^R$ & $ =$ & $ \{C6, -C8\}$& (Loop back for $C6^L$)\\
    \hline
    $-C8^L$ & $ =$ & $ \{C2, C3\}$& (Loop back for $C2^R$ and $C3^R)$
\end{tabular}
\end{center}
The looping back that occurs at boards $C2$, $C3$, and $C6$ ensures that gameplay will be limited to the 13 boards that are listed. Further, since players always have a safe option on the boards where they play next, neither player will be able to force a win, and the game will go on indefinitely.  A similar argument holds when Gray makes the first move. Thus, $o(C0)=\D$, and $2\times 2\times 3$ Abalone is weakly solved.

\end{proof}

\section{Future Work}
\label{sec:futurework}

Strong solutions for all Abalone boards larger than $2\times2\times2$ are still unknown. However, we do have conjectures regarding weak solutions for $2\times3\times3$ and $3\times3\times3$ Abalone. These variants have the same general rules as standard Abalone, except to win a player only needs to push off two of their opponent's pieces. Our initial proposed starting boards for these variants are shown in Figure~\ref{fig:starts}.  Note that $3\times3\times3$ Abalone would use a daisy configuration.

\begin{figure}[ht!]
\caption{Initial configurations for $2\times3\times3$ Abalone (left) and $3\times3\times3$ Abalone (right).}
\label{fig:starts}
\begin{minipage}[c]{0.5\linewidth}
\centering
    \begin{tikzpicture} [hexa/.style= {shape=regular polygon,
                                   regular polygon sides=6,
                                   minimum size=1cm, draw,
                                   inner sep=0,anchor=south,}]
    \foreach \j in {0,...,4}{%
    \ifodd\j 
         \foreach \i in {1,...,3}{\node[hexa] at ({\j/2+\j/4},{(\i+1/2)*sin(60)}) {};}
    \else
        \ifthenelse{\j=0 \or \j=4}{\foreach \i in {2,3}{\node[hexa] at ({\j/2+\j/4},{\i*sin(60)}) {};}}{\foreach \i in {1,...,4}{\node[hexa] at ({\j/2+\j/4},{\i*sin(60)}) {};}}
    
    \fi}
    
    \draw[fill=black] (0,{(2+1/2)*sin(60)}) circle (9pt);
    \draw[fill=black] (0,{(3+1/2)*sin(60)}) circle (9pt);
    \draw[fill=black] (.75,{(2)*sin(60)}) circle (9pt);
    \draw[fill=black] (.75,{(3)*sin(60)}) circle (9pt);
    \draw[fill=black] (.75,{(4)*sin(60)}) circle (9pt);
    
    \draw[fill=lightgray] (2.25,{(2)*sin(60)}) circle (9pt);
    \draw[fill=lightgray] (2.25,{(3)*sin(60)}) circle (9pt);
    \draw[fill=lightgray] (2.25,{(4)*sin(60)}) circle (9pt);
    \draw[fill=lightgray] (3,{(2+1/2)*sin(60)}) circle (9pt);
    \draw[fill=lightgray] (3,{(3+1/2)*sin(60)}) circle (9pt);
    
\end{tikzpicture}
\end{minipage}
\begin{minipage}[c]{0.5\linewidth}
\rotatebox{60}{
\begin{tikzpicture} [hexa/.style= {shape=regular polygon,
                                   regular polygon sides=6,
                                   minimum size=1cm, draw,
                                   inner sep=0,anchor=south,}]
    \foreach \j in {0,...,4}{%
    \ifodd\j 
         \foreach \i in {0,...,3}{\node[hexa] at ({\j/2+\j/4},{(\i+1/2)*sin(60)}) {};}
    \else
        \ifthenelse{\j=0 \or \j=4}{\foreach \i in {1,...,3}{\node[hexa] at ({\j/2+\j/4},{\i*sin(60)}) {};}}{\foreach \i in {0,...,4}{\node[hexa] at ({\j/2+\j/4},{\i*sin(60)}) {};}}
    
    \fi}
    
    \draw[fill=lightgray] (0,{(3+1/2)*sin(60)}) circle (9pt);
    \draw[fill=lightgray] (.75,{(3)*sin(60)}) circle (9pt);
    \draw[fill=lightgray] (0,{(2.5)*sin(60)}) circle (9pt);
    
    \draw[fill=black] (.75,{(1)*sin(60)}) circle (9pt);
    \draw[fill=black] (1.5,{(.5)*sin(60)}) circle (9pt);
    \draw[fill=black] (1.5,{(1.5)*sin(60)}) circle (9pt);

    \draw[fill=black] (2.25,{(4)*sin(60)}) circle (9pt);
    \draw[fill=black] (1.5,{(4.5)*sin(60)}) circle (9pt);
    \draw[fill=black] (1.5,{(3.5)*sin(60)}) circle (9pt);
    
    \draw[fill=lightgray] (3,{(1+1/2)*sin(60)}) circle (9pt);
    \draw[fill=lightgray] (2.25,{(2)*sin(60)}) circle (9pt);
    \draw[fill=lightgray] (3,{(2+1/2)*sin(60)}) circle (9pt);
    
\end{tikzpicture}
}
\end{minipage}
\end{figure}
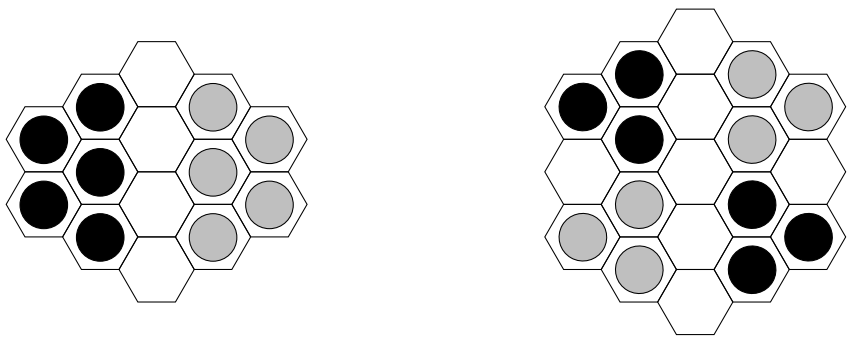

After some experimenting, we developed the following two conjectures.
\begin{conjecture}
    If $D0$ denotes the initial constellation of $2\times 3\times 3$ Abalone, then $o(D0) = \D$.
\end{conjecture}
\begin{conjecture}\label{conj:333}
    If $E0$ denotes the initial constellation of $3\times3\times3$ Abalone, then $o(E0)=\N$.
\end{conjecture}

We were surprised by the possibility that this $3\times 3\times 3$ Abalone board could be won by the first player to make a move.  Our reason for conjecturing so is as follows: Suppose Black moves first.  They can immediately move two marbles into middle spaces in a broadside motion. (See Figure~\ref{fig:333start}.) Taking this option offers Black a position that separates Gray's pieces, and we were unable to find a Gray response to this move that Black cannot counter.  If Conjecture~\ref{conj:333} is correct, then we would need to select a more balanced starting configuration than $E0$ for $3\times 3\times 3$ Abalone.

\begin{figure}[ht!]
\caption{A strong first option for Black in $3\times3\times3$ Abalone.}
\label{fig:333start}
\centering
\rotatebox{60}{
\begin{tikzpicture} [hexa/.style= {shape=regular polygon,
                                   regular polygon sides=6,
                                   minimum size=1cm, draw,
                                   inner sep=0,anchor=south,}]
    \foreach \j in {0,...,4}{%
    \ifodd\j 
         \foreach \i in {0,...,3}{\node[hexa] at ({\j/2+\j/4},{(\i+1/2)*sin(60)}) {};}
    \else
        \ifthenelse{\j=0 \or \j=4}{\foreach \i in {1,...,3}{\node[hexa] at ({\j/2+\j/4},{\i*sin(60)}) {};}}{\foreach \i in {0,...,4}{\node[hexa] at ({\j/2+\j/4},{\i*sin(60)}) {};}}
    
    \fi}
    
    \draw[fill=lightgray] (0,{(3+1/2)*sin(60)}) circle (9pt);
    \draw[fill=lightgray] (.75,{(3)*sin(60)}) circle (9pt);
    \draw[fill=lightgray] (0,{(2.5)*sin(60)}) circle (9pt);
    
    \draw[fill=black] (.75,{(1)*sin(60)}) circle (9pt);
    \draw[fill=black] (1.5,{(.5)*sin(60)}) circle (9pt);
    \draw[fill=black] (1.5,{(1.5)*sin(60)}) circle (9pt);

    \draw[fill=black] (2.25,{(3)*sin(60)}) circle (9pt);
    \draw[fill=black] (1.5,{(4.5)*sin(60)}) circle (9pt);
    \draw[fill=black] (1.5,{(2.5)*sin(60)}) circle (9pt);
    
    \draw[fill=lightgray] (3,{(1+1/2)*sin(60)}) circle (9pt);
    \draw[fill=lightgray] (2.25,{(2)*sin(60)}) circle (9pt);
    \draw[fill=lightgray] (3,{(2+1/2)*sin(60)}) circle (9pt);
    
\end{tikzpicture}
}
\end{figure}
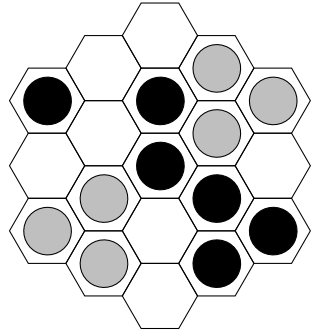

The strategy by which we proved Theorem~\ref{thm:223weaksolve} and developed Conjecture~\ref{conj:333} hints at a more elegant approach to solving further Abalone variants.  Even with computer assistance, an exhaustive analysis of each nonisomorphic constellation would quickly exceed the computing power available today.  Instead, it may be more productive for researchers to focus their efforts on identifying winning subconfigurations.  We expect that many of these subconfigurations will contain forks or convex clusters of pieces in the middle of the board.  For example, consider the families ``Diamond'' and ``Trapezoid'' of $2\times3\times3$ Abalone constellations, which are sets of constellations which contain the subconfigurations shown in Figure~\ref{fig:traps}.  We conjecture that Black has a winning strategy for boards of these types.

\begin{figure}[ht!]
\caption{The ``Diamond'' (left) and ``Trapezoid'' (right) subconfigurations in $2\times3\times3$ Abalone.}
\label{fig:traps}
\begin{minipage}[c]{0.5\linewidth}
\centering
\vspace{.5cm}
    \begin{tikzpicture} [hexa/.style= {shape=regular polygon,
                                   regular polygon sides=6,
                                   minimum size=1cm, draw,
                                   inner sep=0,anchor=south,}]
    \foreach \j in {0,...,4}{%
    \ifodd\j 
         \foreach \i in {1,...,3}{\node[hexa] at ({\j/2+\j/4},{(\i+1/2)*sin(60)}) {};}
    \else
        \ifthenelse{\j=0 \or \j=4}{\foreach \i in {2,3}{\node[hexa] at ({\j/2+\j/4},{\i*sin(60)}) {};}}{\foreach \i in {1,...,4}{\node[hexa] at ({\j/2+\j/4},{\i*sin(60)}) {};}}
    
    \fi}
    
    \draw[fill=black] (1.5,{(2+1/2)*sin(60)}) circle (9pt);
    \draw[fill=black] (1.5,{(3+1/2)*sin(60)}) circle (9pt);
    \draw[fill=black] (.75,{(3)*sin(60)}) circle (9pt);
    \draw[fill=black] (2.25,{(3)*sin(60)}) circle (9pt);
    
\end{tikzpicture}
\end{minipage}
\begin{minipage}[c]{0.4\linewidth}
\centering
\vspace{0.5cm}
\begin{tikzpicture} [hexa/.style= {shape=regular polygon,
                                   regular polygon sides=6,
                                   minimum size=1cm, draw,
                                   inner sep=0,anchor=south,}]
    \foreach \j in {0,...,4}{%
    \ifodd\j 
         \foreach \i in {1,...,3}{\node[hexa] at ({\j/2+\j/4},{(\i+1/2)*sin(60)}) {};}
    \else
        \ifthenelse{\j=0 \or \j=4}{\foreach \i in {2,3}{\node[hexa] at ({\j/2+\j/4},{\i*sin(60)}) {};}}{\foreach \i in {1,...,4}{\node[hexa] at ({\j/2+\j/4},{\i*sin(60)}) {};}}
    
    \fi}
    
    \draw[fill=black] (1.5,{(2+1/2)*sin(60)}) circle (9pt);
    \draw[fill=black] (1.5,{(3+1/2)*sin(60)}) circle (9pt);
    \draw[fill=black] (.75,{(2)*sin(60)}) circle (9pt);
    \draw[fill=black] (.75,{(3)*sin(60)}) circle (9pt);
    \draw[fill=black] (.75,{(4)*sin(60)}) circle (9pt);
    
\end{tikzpicture}
\end{minipage}
\end{figure}
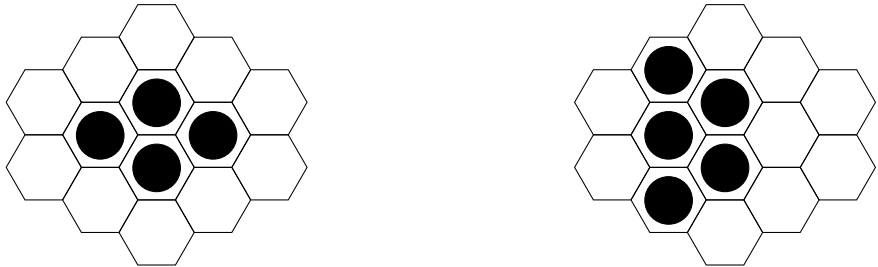

\begin{conjecture}
    If $C$ is a $2\times 3\times 3$ Abalone constellation in the ``Diamond'' or ``Trapezoid'' family, then $o(C)=\LL$.
\end{conjecture}

\section{Acknowledgements}
\label{sec:acknowledgements}

This research was completed during Summer 2019 while Gutstadt and Koerner were undergraduate students and Hogenson was visiting faculty at Colorado College.  The authors would like to thank the Colorado College Office of the Provost for supporting Gutstadt and Koerner via a Faculty Student Collaborative Research (SCoRe) grant.  In particular, we recognize the support and interest of Lisa Schwartz, who was the SCoRe Program Coordinator at the time.  We also extend our thanks to the Colorado College Department of Mathematics and Computer Science for allowing us to use their facilities.


\typeout{}

\bibliography{references}


\appendix
\section{Author Biographies}
\label{appendix:bios}

\textbf{Joseph Gutstadt} received a BA in Mathematics from Colorado College in 2020.  He now works as a Data Engineer in Med Tech, and is currently pursuing an MS in computer science at Georgia Tech.\\ 

\noindent\textbf{John Koerner} received a BA in Mathematical Economics from Colorado College in 2020 and a BS in Applied Mathematics from Columbia University in 2022. He now works in the renewable energy industry and is pursuing an MS in Applied Mathematics at Columbia University.

\end{document}